\documentclass[12pt,a4paper]{amsart}
\pdfoutput=1
\usepackage{amsfonts}
\usepackage{amsthm}
\usepackage{amsmath}
\usepackage[normalem]{ulem}
\usepackage{amscd}
\usepackage[latin2]{inputenc}
\usepackage{t1enc}
\usepackage[mathscr]{eucal}
\usepackage{indentfirst}
\usepackage{graphicx}
\usepackage{graphics}
\usepackage{pict2e}
\usepackage{epic}
\numberwithin{equation}{section}
\usepackage[margin=2.9cm]{geometry}
\usepackage{epstopdf} 
\usepackage{hyperref}
\usepackage{mathtools}
\usepackage{amssymb}
\usepackage{empheq}
\newcounter{cnstcnt}

\hypersetup{
    colorlinks=true,
    linkcolor=blue,
    filecolor=magenta,      
    urlcolor=cyan,
}

\usepackage[
backend=biber,
style=numeric,
sorting=nyvt,
giveninits=true
]{biblatex}
\DeclareFieldFormat[article,periodical]{number}{\bibstring{number}~#1}

\renewbibmacro*{journal+issuetitle}{%
  \usebibmacro{journal}%
  \setunit*{\addspace}%
  \iffieldundef{series}
    {}
    {\newunit
     \printfield{series}%
     \setunit{\addspace}}%
  \printfield{volume}%
  \setunit{\addspace}%
  \usebibmacro{issue+date}%
  \setunit{\addcomma\space}%
  \printfield{number}%
  \setunit{\addcolon\space}%
  \usebibmacro{issue}%
  \setunit{\addcomma\space}%
  \printfield{eid}
  \newunit}

\renewbibmacro{in:}{}

\DeclareFieldFormat{pages}{#1}
\DeclareFieldFormat[article]{number}{\bibstring{number}\addnbspace #1}

\theoremstyle{plain}
\newtheorem{Th}{Theorem}[section]
\newtheorem{Lemma}[Th]{Lemma}
\newtheorem{Cor}[Th]{Corollary}
\newtheorem{Prop}[Th]{Proposition}

 \theoremstyle{definition}

\newtheorem{Conj}[Th]{Conjecture}

\newtheorem{?}[Th]{Problem}

\AtBeginBibliography{\footnotesize}

\AtEndDocument{\bigskip{\footnotesize%
  \textsc{Department of Mathematics, Northwestern University, 2033 Sheridan Road, Evanston, IL 60208} \par  
  }}

\title{Convergence of cscK metrics on smooth minimal models of general type}
\author{Wanxing Liu}
\begin{document}

\begin{abstract}
    We consider constant scalar curvature K\"{a}hler metrics on a smooth minimal model of general type in a neighborhood of the canonical class, which is the perturbation of the canonical class by a fixed K\"{a}hler metric. We show that sequences of such metrics converge smoothly on compact subsets away from a subvariety to the singular K\"{a}hler Einstein metric in the canonical class. This confirms partially a conjecture of Jian-Shi-Song about the convergence behavior of such sequences.
\end{abstract}

\maketitle

\section{Introduction}
 Since the work of Calabi \cite{MR645743, MR780039} there has been much interest in the existence of constant scalar curvature K\"{a}hler (cscK) metrics. For the K\"{a}hler Einstein metric which is a special type of the cscK metric, Yau \cite{yau1978ricci} and Aubin \cite{MR494932} established independently the existence of such a metric on K\"{a}hler manifolds of negative first Chern class, and Yau  \cite{yau1978ricci} also showed it for manifolds of zero first Chern class. For manifolds of positive first Chern class, the Yau-Tian-Donaldson conjecture predicted that the existence of a K\"{a}hler Einstein metric is equivalent to the K-stability. Chen-Donaldson-Sun \cite{MR3264766, MR3264767,MR3264768} proved that the K-stability is sufficient for the existence of a K\"{a}hler Einstein metric (see also Tian \cite{MR3352459}), while the necessity was shown by Tian \cite{MR1471884}, Donaldson \cite{MR1916953}, Stoppa \cite{MR2518643} and Berman \cite{MR3461370}. 

 For the cscK metric, Donaldson \cite{MR2507220} considered the existence of cscK metrics on toric surfaces. Chen-Cheng \cite{1801.00656} proved that the properness of the Mabuchi functional is sufficient for the existence of a cscK metric building on the work of \cite{darvas2017mabuchi, MR3406499, MR3858468}. The necessity was proven by Berman-Darvas-Lu \cite{1602.03114}. We refer the interested reader to \cite{MR3076065, 1807.00928, MR3600039, MR2743451} for surveys and related developments of this area. Following the breakthrough made by Chen-Cheng, Jian-Shi-Song \cite{MR3981128} showed the following theorem.

\begin{Th} \label{existencetheorem}
Let $(M, \omega_0)$ be a compact K\"{a}hler manifold. If the canonical bundle $K_{M}$ is semi-ample, then for any $\varepsilon > 0$ small enough, there exists a unique cscK metric in the K\"{a}hler class  $-c_1(M) + \varepsilon [\omega_0] = c_1(K_{M}) + \varepsilon [\omega_0]$. 
\end{Th}
Recently, using different tools, Dyrefelt \cite{2004.02832} and Song \cite{2012.07956} strengthened this result to all compact K\"{a}hler manifolds with nef canonical bundle. We define the first Chern class $c_1(M)$ to be
\begin{equation}
    c_1(M) = [\mathrm{Ric}(\omega_0)] = -[\sqrt{-1} \partial \bar \partial \log \mathrm{det}(g_0)],
\end{equation}
where $g_0$ is the metric tensor of $\omega_0$. Notice that this differs from the usual definition of $c_1(M)$ by a factor of $2 \pi$. Theorem \ref{existencetheorem} is a generalization of Arezzo-Pacard's \cite{MR2275832} result on minimal surfaces of general type. The proof of this result is based on Chen-Cheng \cite{1801.00656} and the properness criterions developed by Weinkove \cite{MR2226957}, Song-Weinkove \cite{MR2368374} and Li-Shi-Yao \cite{MR3561959},
and Jian-Shi-Song also made the following conjecture.
\begin{Conj} \label{Conj}
Let $(M, \omega_0)$ be a compact K\"{a}hler manifold with semi-ample canonical bundle $K_{M}$. Then any sequence of cscK metrics in $-c_1(M) + \varepsilon [\omega_0]$ converges to the twisted K\"{a}hler-Einstein metric $g_{\text{can}}$ on the canonical model $M_{\text{can}}$ of $M$. The convergence should be both global in Gromov-Hausdorff topology and local in smooth topology away from the singular fibres of the canonical map $\Phi : M \to M_{\text{can}}$.
\end{Conj}

Jian-Shi-Song further noted that this conjecture can be understood from the perspective of slope stability introduced by Ross-Thomas \cite{MR2219940} and is related to the Gross-Wilson \cite{MR1863732} from the standpoint of the Strominger-Yau-Zaslow \cite{MR1429831} conjecture in mirror symmetry. If the canonical model of $M$ is smooth and the canonical morphism $\Phi: M \to M_{\text{can}}$ has no singular fibers,  Conjecture \ref{Conj} was shown to be true by Fine \cite{MR2144537} (Theorem 8.1 and its proof), and the main purpose of the current paper is to show that the local smooth convergence of Conjecture $\ref{Conj}$ holds when $-c_1(M)$ is big and nef. 

We call a compact K\"{a}hler manifold $M$ a smooth minimal model if $-c_1(M)$ is nef, and a manifold of general type if $-c_1(M)$ is big. Recall that $-c_1(M)$ is said to be nef if for all $\varepsilon > 0$, $-c_1(M) > \varepsilon [\omega_0]$ and big and nef if it is nef and $(-c_1(M))^{n} > 0$. One immediate benefit of assuming that $-c_1(M)$ is big and nef is the existence of a semi-positive representative in $-c_1(M)$. Since $-c_1(M)$ is big and nef, $M$ is Moishezon which implies that $M$ is projective. By Kawamata's base point free theorem, $K_{M}$ is semi-ample. Hence there exists $\eta \in -c_1(M)$, which is some multiple of the pullback of the Fubini-Study metric through the canonical map $\Phi : M \to \mathbb P^{N}$. In particular, $\eta$ is semi-positive, and we consider the sequence of reference K\"{a}hler metrics $\omega_{\varepsilon} = \eta + \varepsilon \omega_0 > 0$. Also notice that we can choose a volume form $\Omega$ such that $\mathrm{Ric}(\Omega) = -\eta$, then by Yau's theorem \cite{yau1978ricci} we can always choose $\omega_0$ to be such that $\text{Ric}(\omega_0) = - \eta$.

\begin{Th} \label{wode}
Suppose that $(M, \omega_0)$ is a compact K\"{a}hler manifold of dimension $n$ with its canonical bundle being big and nef, and $\mathrm{Ric}(\omega_0) = - \eta$. There exists an effective divisor $E$ such that the sequence of cscK metrics $\omega_{\varphi_{\varepsilon}} \in -c_1(M) + \varepsilon [\omega_0]$, given by Theorem \ref{existencetheorem}, converges in $C_{\mathrm{loc}}^{\infty}(M \backslash E)$ to the unique singular K\"{a}hler Einstein metric in $-c_1(M)$ as $\varepsilon \to 0$.
\end{Th}

 The singular K\"{a}hler Einstein metric was first constructed by Kobayashi \cite{MR799669} in complex dimension 2, and then it was constructed in all dimensions as the limit of the normalized K\"{a}hler Ricci flow. Consider the normalized K\"{a}hler-Ricci flow:

\begin{equation}
    \frac{\partial}{\partial t} \omega_{t} = -\text{Ric}(\omega_{t}) - \omega_{t}, \ \omega|_{t = 0} = \omega_0, \ \omega_{t} > 0.
\end{equation}
The following theorem first appeared in the work of Tsuji \cite{Tsuji1988}, but later Tian-Zhang \cite{tian2006kahler} extended it and clarified some parts of the proof (see also Song-Weinkove \cite{MR3185333} or Tosatti \cite{MR3831026} for nice expositions of this result).

\begin{Th} \label{singular}
Let $M$ be a smooth minimal model of general type. Then
\begin{enumerate}
    \item The solution $\omega = \omega(t)$ of the normalized K\"{a}hler-Ricci flow starting at any K\"{a}hler metric $\omega_0$ on $M$ exists for all time.
    
    \item There exists an effective divisor $E$ of $M$ such that $\omega(t)$ converges in $C^{\infty}_{\mathrm{loc}}(M \backslash E)$ to a smooth K\"{a}hler metric on $M \backslash E$.
\end{enumerate}

Furthermore, the K\"{a}hler metric obtained above is the unique metric satisfying
\begin{enumerate}
    \item $\mathrm{Ric}(\omega_{\mathrm{KE}}) = -\omega_{\mathrm{KE}}$ on $M \backslash E$.
    \item There exists a constant $C$ such that 
    \begin{equation}
        C \omega_0^{n} \leq \omega_{\mathrm{KE}}^{n} \leq \frac{1}{C} \omega_0^{n}.
    \end{equation}
    
\end{enumerate}
\end{Th}
$\omega_{\text{KE}}$ was also constructed by Eyssidieux-Guedj-Zeriahi \cite{MR2505296} using pluripotential theory.

  As a consequence of Theorem \ref{existencetheorem}, we can pick a unique sequence of cscK metrics $\omega_{\varphi_{\varepsilon}} \in [\omega_{\varepsilon}]$. The K\"{a}hler potentials of these metrics $\varphi_{\varepsilon} \in \mathcal{H_{\omega_{\varepsilon}}}$ satisfy the following coupled equations:

\begin{equation} \label{fundamental}
    \begin{aligned} 
    \frac{\omega_{\varphi_{\varepsilon}}^n}{\omega_0^n} &= e^{F_{\varepsilon}},\\
    \Delta_{\omega_{\varphi_{\varepsilon}}} F_{\epsilon} &= - \underline{R_{\varepsilon}} - \mathrm{tr}_{\omega_{\varphi_{\varepsilon}}} \eta
    \end{aligned}
\end{equation}
where $\underline{R_{\epsilon}} = n \frac{c_1(M) \cdot [\omega_{\epsilon}]^{n-1}}{ [\omega_{\varepsilon}]^{n}}$, $\underline{R_{\epsilon}} \to -n$ as $\epsilon \to 0$, and 
\begin{equation}
 \mathcal{H}_{\omega_{\varepsilon}} = \{v \in C^{\infty}(M) |\omega_{\varepsilon} + \sqrt{-1} \partial \bar \partial v > 0, \sup_{M} v = 0\}.
\end{equation}

We stress that it is important to assume bigness of $-c_1(M)$ in order for the limit of $\underline{R_{\varepsilon}}$ to be $-n$, otherwise it is not true. Also notice that the equations (\ref{fundamental}) are slightly different from the coupled equations considered in \cite{1712.06697}:
\begin{equation}
    \begin{aligned}
      \frac{\omega_{\varphi_{\varepsilon}}^n}{\omega_{\varepsilon}^n} &= e^{F_{\varepsilon}},\\
    \Delta_{\omega_{\varphi_{\varepsilon}}} F_{\epsilon} &= - \underline{R_{\varepsilon}} + \mathrm{tr}_{\omega_{\varphi_{\varepsilon}}} \mathrm{Ric(\omega_{\varepsilon})}.
    \end{aligned}
\end{equation}
Specifically, we replace $\omega_{\varepsilon}$ in the denominator of the first equation by $\omega_0$ and adapt the second equation accordingly because $\omega_{\varepsilon}$ is degenerating. We will see later that we will have to adjust the definition of the Mabuchi Energy to accommodate this change. The strategy of the proof of Theorem \ref{wode} consists of the following steps:
\begin{enumerate}
    \item We establish through sections \ref{firstSection} - \ref{fourthsection} a degenerate version of the estimates in \cite{1712.06697}. More specifically, we show that $\varphi_{\varepsilon}$ is bounded in $C^{\infty}_{\text{loc}}(M \backslash E)$ and the bound depends only on $\int_{M} \mathrm{log} \frac{\omega_{\varphi_{\varepsilon}}^n}{\omega_0^n} \frac{\omega_{\varphi_{\varepsilon}}^n}{n!} =  \int_{M} e^{F_{\varepsilon}} F_{\varepsilon} \frac{\omega_0^n}{n!}$ and $(M, \omega_0)$. We will call $\int_{M} \mathrm{log} \frac{\omega_{\varphi_{\varepsilon}}^n}{\omega_0^n} \frac{\omega_{\varphi_{\varepsilon}}^n}{n!}$ the entropy.
    
    \item We show in section \ref{section5} using a method of Dervan \cite{AFST_2016_6_25_4_919_0} that $\int_{M} e^{F_{\varepsilon}} F_{\varepsilon}  \frac{\omega_0^n}{n!}$ is uniformly bounded i.e. independent of $\varepsilon$, thus $\varphi_{\varepsilon}$ is uniformly bounded in $C^{\infty}_{\text{loc}}(M \backslash E)$.
    
    \item In section \ref{lastsection} we use the estimates and an integration by part argument to conclude that $\omega_{\varphi_{\varepsilon}}$ has to converge in $C^{\infty}_{\text{loc}}(M \backslash E)$ to the unique singular K\"{a}hler Einstein metric in $-c_1(M)$.
\end{enumerate}

We remark that Zheng \cite{1803.09506} considered the problem of the $L^{1}$ convergence of a degenerating  sequence of smooth cscK metrics in a neighborhood of an arbitrary big class. There he also had to generalize Chen-Cheng's \cite{1712.06697} original estimates. Our estimates are different in the sense that we are able to get full non-degenerate $0$-th order estimates on $F_{\varepsilon}$. Furthermore, the estimates in Zheng \cite{1803.09506} are proved with respect to a sequence of specially constructed reference metrics. Also, one critical element of his proof is a version of alpha invariant for any big class using machinery from pluripotential theory, but we do not need it here.

We conclude this section by mentioning that there are also generalizations of Chen-Cheng's estimates in other directions. For instance, Shen \cite{1909.13445} generalizes them to the Hermitian setting, and He generalizes them to Sasaki manifolds \cite{1802.03841}, and extremal metrics  \cite{MR4014289}.

\section{
  \texorpdfstring{$C^{0}$}{TEXT} estimates} \label{firstSection}
In this section we produce $C^0$ estimates for $\varphi_{\varepsilon}$ and $F_{\varepsilon}$ solving the following coupled equations:

\begin{equation} \label{maineq}
    \begin{aligned}
    \frac{\omega_{\varphi_{\varepsilon}}^n}{\omega_0^n} &= e^{F_{\varepsilon}},\\
    \Delta_{\omega_{\varphi_{\varepsilon}}} F_{\epsilon} &= - \underline{R_{\varepsilon}} - \mathrm{tr}_{\omega_{\varphi_{\varepsilon}}} \eta.
    \end{aligned}
\end{equation}
We start with the $C^0$ estimate on $\varphi_{\varepsilon}$.
\begin{Th} \label{mainprop}
There exists a constant $C > 0$ such that $\|\varphi_{\varepsilon}\|_{C^0(M)} \leq C$ where $C$ is dependent on $\int_{M} \mathrm{log} \frac{\omega_{\varphi_{\varepsilon}}^n}{\omega_0^n} \frac{\omega_{\varphi_{\varepsilon}}^n}{n!}$. 
\end{Th}
Before proving this, we need some preparations. Denote $\text{Vol}(\omega_{\varepsilon}) = \int_{M} \omega_{\varepsilon}^n$, and as in \cite{1712.06697} we define a smooth real-valued function $\rho_{\varepsilon}$ to be the unique solution to the following equation using Yau's theorem \cite{yau1978ricci}:

\begin{equation} \label{2.2}
\begin{aligned}
    ( \omega_{\varepsilon} + \sqrt{-1} \partial \bar \partial \rho_{\varepsilon})^n &= \frac{\text{Vol}(\omega_{\varepsilon})e^{F_{\varepsilon}} \Phi(F_{\varepsilon}) \omega_0^n}{\int_{M} e^{F_{\varepsilon}} \Phi(F_{\varepsilon}) \omega_0^n}, \\
    \sup_{M} \rho_{\varepsilon} &= 0
\end{aligned}
\end{equation}
where $\Phi(F_{\varepsilon}) = \sqrt{F_{\varepsilon}^2 + 1}$. Let us then recall the definition of Tian's $\alpha$-invariant \cite{Tian1987}.

\begin{Prop}[Tian's $\alpha$-invariant]
For any K\"{a}hler class $[\omega]$ on $M$, there exists an invariant $\alpha(M, [\omega]) > 0$ such that for any $\alpha < \alpha(M, [\omega])$, we have
\begin{equation}
    \int_{M} e^{-\alpha v} \omega^n  \leq C
\end{equation}
\begin{flushleft}
for all $v \in \mathcal{H}_{\omega}$.
\end{flushleft}
\end{Prop}

\begin{Lemma} \label{alphainvariant}
Given $\alpha > 0$ with $\alpha < \alpha(M, [\eta + \omega_0])$, there is a uniform constant $C > 0$ such that for all $\varepsilon > 0$ with $ \varepsilon \leq 1$ we have
\begin{equation}
    \int_{M} e^{-\alpha v} \omega_0^n  \leq C
\end{equation}

\begin{flushleft}
for all $v \in \mathcal{H}_{\eta + \varepsilon \omega_0}$.
\end{flushleft}

\end{Lemma}

\begin{proof}
The key observation is that when $\varepsilon > 0$ and $\varepsilon \leq 1$, $\eta + \varepsilon \omega_0 \leq \eta + \omega_0$  implies that $\mathcal{H}_{\eta + \varepsilon \omega_0} \subset \mathcal{H}_{\eta + \omega_0}$. Thus, for any $v \in \mathcal{H}_{\eta + \varepsilon \omega_0}$, we have for the fixed $\alpha > 0$ and $\alpha < \alpha(M, [\eta + \omega_0])$
\begin{equation}
    \begin{aligned}
    \int_{M} e^{- \alpha v} \omega_0^n \leq \int_{M} e^{- \alpha v} (\eta + \omega_0)^n \leq C.
    \end{aligned}
\end{equation}
\end{proof}
The most important estimate we need is the following Lemma analogous to Theorem 5.2 in \cite{1712.06697}.

\begin{Lemma} \label{maintool}
There exists a constant $C$ such that for all $\varepsilon > 0$ sufficiently small 
\begin{equation}
    F_{\varepsilon} + \varepsilon \rho_{\varepsilon} - (1 + \varepsilon) \varphi_{\varepsilon} \leq C
\end{equation}
where $C$ only depends on $\int_{M} \mathrm{log} \frac{\omega_{\varphi_{\varepsilon}}^n}{\omega_0^n} \frac{\omega_{\varphi_{\varepsilon}}^n}{n!}$.
\end{Lemma}

\begin{proof}
Given a point $p_0 \in M$, $0 < d_0 < 1$, we start by choosing a smooth cut-off function $f$ such that
\begin{equation}
    \begin{aligned}
    &1 - \theta \leq f \leq 1, \\
    &f(p_0) = 1, f \equiv 1 - \theta \text{ outside } B_{\frac{d_0}{2}}(p_0),\\
     &|\partial f|_{\omega_0}^2 \leq \frac{4\theta ^2}{d_0^2}, |\partial^2 f|^2_{\omega_0} \leq \frac{4\theta}{d_0^2}
\end{aligned}
\end{equation}
where $p_0$ and $\theta > 0$ are going to be specified later, and $d_0$ is fixed to be a constant strictly larger than 0 and strictly less than 1. Let $\alpha > 0$ be a fixed constant strictly less than $\alpha(M, [\eta + \omega_0])$, and choose $\delta = \frac{\alpha}{4n}$. We will denote the metric tensor of $\omega_{\varphi_{\varepsilon}}$ by $g_{\varphi_{\varepsilon}}$, and suppress $\varepsilon$ for simplicity of notation while carrying out calculations. Calculate 
\begin{equation} \label{fund}
\begin{aligned}
    &\Delta_{\omega_\varphi}(e^{\delta(F + \varepsilon \rho - (1 + \varepsilon) \varphi)}f) \\
    &= \Delta_{\omega_\varphi}(e^{\delta(F + \varepsilon \rho - (1 + \varepsilon) \varphi)}) f + e^{\delta(F + \varepsilon \rho - (1 + \varepsilon) \varphi)} \Delta_{\omega_{\varphi}}(f)  \\
   & +  2 e^{\delta(F + \varepsilon \rho - (1 + \varepsilon) \varphi)} \delta Re(g^{i \bar j}_{\varphi} \partial_{i} (F + \varepsilon \rho - (1 + \varepsilon) \varphi ) \bar \partial_{j} f)\\
   &=  e^{\delta(F + \varepsilon \rho - (1 + \varepsilon) \varphi )} f(\delta^2|\partial (F + \varepsilon \rho - (1 + \varepsilon) \varphi)|_{\omega_{\varphi}}^2 
   + \delta \Delta_{\omega_{\varphi}}(F + \varepsilon \rho - (1 + \varepsilon) \varphi)) \\
   &+ e^{\delta(F + \varepsilon \rho - (1+\varepsilon) \varphi)} \Delta_{\omega_\varphi} f
   +2 e^{\delta(F + \varepsilon \rho - (1+\varepsilon) \varphi)} \delta Re(g^{i \bar j}_{\varphi} \partial_{i} (F + \varepsilon \rho - (1 + \varepsilon) \varphi) \bar \partial_{j} f).
\end{aligned}
\end{equation}
We estimate the terms involved in the above calculation as follows

\begin{equation}
    \begin{aligned}
    &2 \delta Re(g^{i \bar j}_{\varphi} \partial_{i} (F + \varepsilon \rho - ( 1+ \varepsilon) \varphi ) \bar \partial_{j} f) \\
    &\geq - \delta^2 f |\partial(F + \varepsilon \rho - (1 + \varepsilon) \varphi)|_{\omega_{\varphi}}^2 - \frac{|\partial f|_{\omega_{\varphi}}^2}{f}\\
    & \geq -\delta^2 f |\partial(F + \varepsilon \rho - (1 + \varepsilon) \varphi)|_{\omega_{\varphi}}^2 - \frac{|\partial f|_{\omega_0}^2\text{tr}_{\omega_{\varphi}}\omega_0}{f}\\
    & \geq -\delta^2 f |\partial(F + \varepsilon \rho - (1 + \varepsilon) \varphi)|_{\omega_{\varphi}}^2 -\frac{4 \theta^2}{d_0^2(1-\theta)}\text{tr}_{\omega_{\varphi}}\omega_0,
    \end{aligned}
\end{equation}

\begin{equation}
    \begin{aligned}
     e^{\delta(F + \varepsilon \rho - (1 + \varepsilon) \varphi)} \Delta_{\omega_{\varphi}} f  \geq  -e^{\delta(F + \varepsilon \rho - (1 + \varepsilon) \varphi)} \frac{4 \theta}{d_0^2(1-\theta)} f\text{tr}_{\omega_{\varphi}}\omega_0.
    \end{aligned}
\end{equation}
The key computation is the following:
\begin{equation}
    \Delta_{\omega_{\varphi}}( F + \varepsilon \rho - (1 + \varepsilon) \varphi)  = -(\underline{R} + (1 + \varepsilon) n) - \text{tr}_{\omega_{\varphi}} \eta + (1 + \varepsilon) \text{tr}_{\omega_{\varphi}}  \omega + \varepsilon \Delta_{\omega_{\varphi}} \rho.
\end{equation}
Let $A_{\Phi}(F) = \int_{M} e^{F} \Phi(F) \omega_0^n$, and notice by (\ref{2.2}) 
\begin{equation}
\begin{aligned}
    \Delta_{\omega_{\varphi}} \rho &= \text{tr}_{\omega_{\varphi}} (\omega + \sqrt{-1} \partial \bar \partial \rho) - \text{tr}_{\omega_{\varphi}}  \omega \\
    & \geq n(e^{-F}e^{F}\text{Vol}(\omega)\Phi(F)A_{\Phi}(F)^{-1})^{\frac{1}{n}} - \text{tr}_{\omega_{\varphi}}  \omega \\
    & = n(\text{Vol}(\omega)\Phi(F) A_{\Phi}(F)^{-1})^{\frac{1}{n}} - \text{tr}_{\omega_{\varphi}} \omega.
\end{aligned}
\end{equation}
Then we have
\begin{equation}
\begin{aligned}
    &\Delta_{\omega_\varphi}( F + \varepsilon \rho - (1 + \varepsilon) \varphi)\\
    &\geq (-\underline{R} - (1+ \varepsilon) n + \varepsilon n \text{Vol}(\omega)^{\frac{1}{n}}A_{\Phi}(F)^{-\frac{1}{n}}\Phi(F)^{\frac{1}{n}})+ (1 + \varepsilon) \text{tr}_{\omega_{\varphi}} \omega \\
    &- \varepsilon \text{tr}_{\omega_{\varphi}} \omega - \text{tr}_{\omega_{\varphi}}\eta\\
    &= (-\underline{R} - (1 + \varepsilon) n + \varepsilon n \text{Vol}(\omega)^{\frac{1}{n}}A_{\Phi}(F)^{-\frac{1}{n}}\Phi(F)^{\frac{1}{n}}) + \varepsilon \text{tr}_{\omega_{\varphi}} \omega_0.
\end{aligned}  
\end{equation}
Combining all these calculations, we conclude that
\begin{equation} \label{key estimate}
    \begin{aligned}
    &\Delta_{\omega_{\varphi}}( e^{\delta(F + \varepsilon \rho - (1 + \varepsilon) \varphi)}f)\\
    &\geq \delta f e^{\delta(F + \varepsilon \rho - (1 + \epsilon) \varphi)}(- \underline{R} - ( 1+ \varepsilon)n + \varepsilon n \text{Vol}(\omega)^{\frac{1}{n}}A_{\Phi}(F)^{-\frac{1}{n}}\Phi(F)^{\frac{1}{n}})\\
    &+  e^{\delta(F + \varepsilon \rho - (1 + \varepsilon) \varphi)}(\delta f \varepsilon - \frac{4 \theta}{d_0^2(1-\theta)} f -\frac{4 \theta ^2}{d_0^2(1-\theta)^2}) \text{tr}_{\omega_{\varphi}}\omega_0.
    \end{aligned}
\end{equation}
Choosing $\theta = \frac{\delta \varepsilon}{64} d_0^2 = \frac{\alpha \varepsilon}{256n} d_0^2$, notice that when $\varepsilon$ is small enough we have $ \theta < \frac{1}{2}$ and $\frac{\delta \varepsilon}{64} \leq  \sqrt{\frac{\delta\varepsilon}{128}}$,  so that because $f \leq 1$ we estimate

\begin{equation}
\begin{aligned}
    \delta f\epsilon- \frac{4 \theta}{d_0^2(1-\theta)} f -\frac{4 \theta ^2}{d_0^2(1-\theta)^2} &\geq \delta(1 - \theta) \epsilon - \frac{4 \theta}{d_0^2 (1-\theta)} - \frac{4 \theta^2}{d_0^2(1-\theta)^2}\\
    & \geq \frac{\delta \epsilon}{2} - \frac{8 \theta}{d_0^2} - \frac{16 \theta^2}{d_0^2} \\
    &\geq \frac{\delta\varepsilon}{2} -  \frac{\delta\varepsilon}{8} -  \frac{\delta\varepsilon}{8}>0.
\end{aligned}
\end{equation}
So the coefficient of $\text{tr}_{\omega_{\varphi}}\omega_0$ in (\ref{key estimate}) is positive, and we throw it away and conclude
\begin{equation}
\begin{aligned}
    &\Delta_{\omega_{\varphi}}( e^{\delta(F + \varepsilon \rho - ( 1+ \varepsilon) \varphi)}f)\\
    &\geq \delta f e^{\delta(F + \varepsilon \rho - (1 + \varepsilon) \varphi)}(- \underline{R} - (1+\varepsilon) n + \varepsilon n \text{Vol}(\omega)^{\frac{1}{n}}A_{\Phi}(F)^{-\frac{1}{n}}\Phi(F)^{\frac{1}{n}}).
\end{aligned}
\end{equation}
Let $u =  e^{\delta(F + \varepsilon \rho - (1 + \varepsilon) \varphi)}$, and choose $p_0$ to be the maximum point of $u$. Applying the Alexandrov-Bakelman-Pucci maximum principle in $B_{d_0}(p_0)$ we get
\begin{equation}
\begin{aligned}
    &\sup_{B_{d_0}(p_0)} u f  \leq \sup_{\partial B_{d_0}(p_0)} u f \\
    &+C {\Bigg(\int_{B_{d_0}(p_0)} \frac{u^{2n}\Big((- \underline{R} - (1 + \varepsilon) n + \varepsilon n \text{Vol}(\omega)^{\frac{1}{n}}A_{\Phi}(F)^{-\frac{1}{n}}\Phi(F)^{\frac{1}{n}})^{-}\Big)^{2n}}{e^{-2F}} \omega_0^{n}\Bigg)}^{\frac{1}{2n}}
\end{aligned}
\end{equation}
where $C$ is a constant dependent on the dimension of the manifold $n, d_0$ and $\delta$. Notice that the integral is only nonzero on the region where
\begin{equation} \label{observation}
    -\underline{R} - (1 + \varepsilon) n + \varepsilon n \text{Vol}(\omega)^{\frac{1}{n}} A_{\Phi}(F)^{-\frac{1}{n}}\Phi(F)^{\frac{1}{n}} < 0.
\end{equation}
Observe that for $\varepsilon$ sufficiently small,
\begin{equation} \label{cancellation}
    \begin{aligned}
    - \underline{R} - (1 + \epsilon) n &= \frac{n [\eta] \cdot \sum_{k = 0}^{n-1}{\binom{n-1}{k}} \varepsilon^{k}[\omega_0]^k \cdot [\eta]^{n-1-k}}{\sum_{k = 0}^n {\binom{n}{k}} \varepsilon^{k}[\omega_0]^k \cdot [\eta]^{n-k}} - (1 + \epsilon) \geq -C\varepsilon.
    \end{aligned}
\end{equation}
So (\ref{observation}) is only possible if $F \leq C$ where $C$ depends only on $A_{\Phi}(F)$, $n$ and
$A_{\Phi}(F)$ depends only on the entropy. So we get

\begin{equation}
    \begin{aligned}
    & \Bigg(\int_{B_{d_0}(p_0)} \frac{u^{2n}\Big(- \underline{R} - (1 + \varepsilon) n + \varepsilon n A_{\Phi}(F)^{-\frac{1}{n}}\Phi(F)^{\frac{1}{n}})^{-}\Big)^{2n}}{e^{-2F}}  \omega_0^n \Bigg)^{\frac{1}{2n}}\\
    & = \Bigg(\int_{B_{d_0}(p_0) \cap \{F \leq C\}} \frac{u^{2n}\Big(- \underline{R} - (1 + \varepsilon) n + \varepsilon n A_{\Phi}(F)^{-\frac{1}{n}}\Phi(F)^{\frac{1}{n}})^{-}\Big)^{2n}}{e^{-2F}}  \omega_0^n \Bigg)^{\frac{1}{2n}}\\
    &\leq \Bigg(\int_{B_{d_0}(p_0)} \frac{u^{2n}((- \underline{R} - (1 + \varepsilon) n)^{-})^{2n}}{e^{-2F}}  \omega_0^n \Bigg)^{\frac{1}{2n}}\\
    & \leq C \Bigg(\int_{B_{d_0}(p_0) \cap \{F \leq C \} }  e^{2n\delta(F + \varepsilon \rho - (1 + \epsilon) \varphi)}e^{2F}\varepsilon^{2n} \omega_0^n \Bigg)^{\frac{1}{2n}}\\
    &\leq C \varepsilon \Big( \int_{B_{d_0}(p_0)} e^{-(2n\delta (1 + \varepsilon))\varphi} \omega_0^n \Big)^{\frac{1}{2n}}\\
    & \leq  C \varepsilon \Big( \int_{B_{d_0}(p_0)} e^{-\frac{\alpha}{2} \varphi} \omega_0^n\Big)^{\frac{1}{2n}}\\
    & \leq C \varepsilon.
    \end{aligned}
\end{equation}
where for the last inequality we used Lemma \ref{alphainvariant}. So

\begin{equation}
    u(p_0) = \sup_{M} u \leq (1-\theta) \sup_{M} u + C \varepsilon, 
\end{equation}
and 
\begin{equation}
    u(p_0) \leq \frac{C \varepsilon }{\theta} = \frac{C \varepsilon }{\frac{\delta \varepsilon d_0^2}{64}}\leq C.
\end{equation}
\end{proof} 
Thus we are able to conclude that for any $\varepsilon$ sufficiently small,
\begin{equation}
    F + \varepsilon \rho - (1 + \varepsilon) \varphi  \leq C.
\end{equation}
With this estimate at our disposal, we now proceed to prove Lemma \ref{mainprop}.

\begin{proof}[Proof of Lemma \ref{mainprop}]
Let $\alpha$ be a positive constant strictly less than $\alpha(M, [\eta + \omega_0])$, the $\alpha$-invariant associated with $[\eta + \omega_0]$, and $\varepsilon < \frac{\alpha}{2}$, we have
\begin{equation}
   F + \frac{\alpha}{2} \rho - (1 + \varepsilon) \varphi \leq F + \varepsilon \rho   - (1 + \varepsilon) \varphi \leq C
\end{equation}
because $\rho \leq 0$. Thus for $\varepsilon < \frac{\alpha}{2}$,

\begin{equation}
    \begin{aligned}
    C \geq \int_{M} e^{-\alpha \rho} \omega_0^n \geq \int_{M} \exp(2(F - (1+ \varepsilon) \varphi - C)) \omega_0^n
    \geq \int_{M} \exp(2(F -C)) \omega_0^n
    \end{aligned}
\end{equation}
where for the last inequality we used the fact that $\varphi \leq 0$. So we conclude that $e^{F} \in L^{2}(M, \omega_0^{n})$ when $\varepsilon$ is small enough. Then the proof is done by applying a theorem first announced in Tian-Zhang \cite{tian2006kahler} and later proved in Zhang \cite{zhang2006degenerate} which asserts that if for some $p > 1$, $e^{F} \in L^{p}(M, \omega_0^n)$, then the $C^{0}$ estimate on $\varphi$ only depends on the $\|e^{F}\|_{L^{p}(M, \omega_0^n)}$. This is a generalization of Kolodziej's \cite{MR1618325} fundamental result to the degenerate setting. One can find more general versions of their result in Eyssidieux-Guedj-Zeriahi \cite{8192341, MR2505296} and  Demailly-Pali \cite{MR2647006} as well.
\end{proof}

To show that $F_{\varepsilon}$ is uniformly bounded from below we recall a trick due to Song-Tian \cite{MR3506382} which was pointed out to the author by Jian Song. We consider the following auxiliary complex Monge-Amp{\'e}re equations:
\begin{equation} \label{auxiliary}
    (\omega_{\varepsilon} + \sqrt{-1} \partial \bar \partial h_{\varepsilon})^n = e^{h_{\varepsilon}} \omega_0^{n}.
\end{equation}
Again, for each $\varepsilon$ there exists a unique smooth real-valued $h_{\varepsilon}$ solving (\ref{auxiliary}) by Yau's theorem. At a maximum $p_0$ of $h_{\varepsilon}$ we have
\begin{equation}
    e^{h_{\varepsilon}} = \frac{(\omega_{\varepsilon} + \sqrt{-1} \partial \bar \partial h_{\varepsilon})^{n}}{\omega_{0}^n} \leq \frac{\omega_{\varepsilon}^n}{\omega_0^n} \leq C.
\end{equation}
Thus, $h_{\varepsilon}$ is uniformly bounded from above. Then by the theorem of Zhang \cite{zhang2006degenerate} again, we have that $\|h_{\varepsilon}\|_{C^{0}(M)}$ is uniformly bounded.

\begin{Th} \label{boundF}
There exists a constant $C$ depending only on $\|\varphi_{\varepsilon}\|_{C^0(M)}$ and $\|h_{\varepsilon}\|_{C^0(M)}$ such that
\begin{equation}
     F_{\varepsilon} \geq C
\end{equation}
for $\varepsilon$ sufficiently small.
\end{Th}
\begin{proof}
Compute at a minimum $p_0$ of $F +  \varphi - h$, we get
\begin{equation}
\begin{aligned}
    0 \leq \Delta_{\omega_{\varphi}}(F + \varphi - h) &=  - \underline{R} +  n - \text{tr}_{\omega_{\varphi}} (\omega + \eta + \sqrt{-1} \partial \bar \partial h) \\
    & \leq C - \text{tr}_{\omega_{\varphi}}(\omega + \sqrt{-1} \partial \bar \partial h)\\
    & \leq C - n e^{\frac{h - F}{n}}.
\end{aligned}
\end{equation}
It implies that
\begin{equation}
    F(p_0) \geq C.
\end{equation}
So for any $p \in M$, 
\begin{equation}
\begin{aligned}
     F(p) + \varphi(p) - h(p) &\geq F(p_0) +  \varphi(p_0) + h(p_0)\\
     &\geq  C +  \varphi(p_0).
\end{aligned}
\end{equation}
\end{proof}

\begin{Th}
There exists a constant $C > 0$ such that
\begin{equation}
    F_{\varepsilon} \leq C
\end{equation}
for all $\varepsilon$ sufficiently small.
\end{Th}

\begin{proof}
Using the $L^p$ bound on $e^{F}$, the lower bound on $F$ and the theorem of Zhang \cite{zhang2006degenerate} again we can show that $\rho_{\varepsilon}$ is uniformly bounded as well. Then the theorem follows immediately from Lemma \ref{maintool}.
\end{proof}

\section{Degenerate bound on \texorpdfstring{$|\partial \varphi_{\varepsilon}|_{\omega_0}^2$}{TEXT}}

Let us first recall the following the Kodaira's lemma, for handling degenerate reference metrics.

\begin{Lemma} \label{Tsuji}
There exists an effective divisor $E$ on $M$, a holomorphic section $s$ of $E$ which vanishes to order 1 along the divisor $E$, constants $\sigma > 0, C > 0$ such that for any $\delta' \in (0, \sigma]$, and any Hermitian metric $h$ of $[E]$, where $[E]$ is the line bundle associated with $E$, we have
\begin{equation}
    \eta +  \delta' \sqrt{-1} \partial \bar \partial \mathrm{log}|s|_{h}^2 > C \delta' \omega_0.
\end{equation}
\end{Lemma}
Moreover, as a special case of the Kodaira lemma, we can and will choose $E$ to be the exceptional locus of the canonical map $\Phi :M \to M_{\mathrm{can}}$, where $\Phi$ fails to be an isomorphism, so that $\eta$ is K\"{a}hler outside $E$ (see Song-Tian \cite{MR3595934} Proposition 2.1). The trick of applying Kodaira's lemma is commonly referred to as \emph{Tsuji's trick} in the literature, and the idea of it is straightforward. $\omega_{\varepsilon}$ is degenerate in the sense that it is tending to $\eta$ which might be zero along $E$. By using the barrier function $\text{log} |s|^{2}_{h}$ which is $-\infty$ along the singular set, one can carry out the usual maximum principle away from $E$. Equipped with this tool, we will show in this section a degenerate version of Theorem 2.1 in \cite{1712.06697}.
\begin{Th} \label{gradientbound222}
There exists constants $q > 0$ and $C > 0$ such that
\begin{equation}
    |\partial \varphi_{\varepsilon}|_{\omega_0}^2 \leq C \frac{1}{|s|_{h}^{2q}} 
\end{equation}
where the constants depend only on $\|\varphi_{\varepsilon}\|_{C^{0}(M)}$, $\|F_{\varepsilon}\|_{C^{0}(M)}$ and $(M, \omega_0)$.
\end{Th}

\begin{proof}
We will suppress $\varepsilon$ for simplicity and apply the maximum principle to 
\begin{equation}
  e^{-(F + \lambda \varphi) + \frac{1}{2}\varphi^2 + Q \delta' \psi}(|\partial \varphi|_{\omega_0}^2 + K).   
\end{equation}
Here $\lambda > 0, K > 0, Q > 0$ are constants to be determined later, and $\delta' \in (0, \min\{\sigma, 1\}]$, where $\sigma$ is the constant given in Lemma \ref{Tsuji}. Let $A(F, \varphi, \psi) = -(F + \lambda  \varphi) + \frac{1}{2} \varphi^2 + Q \delta' \psi$, and denote the metric tensor of $\omega_{\varphi}$ by $g_{\varphi}$, the metric tensor of $\omega$ by $g$ and the metric tensor of $\omega_0$ by $g_0$. Unless otherwise noted, we will always choose a normal holomorphic coordinate neighborhood for $g_0$ such that $g_{\varphi}$ is diagonal.

\begin{equation}
\begin{aligned}
    &\Delta_{\omega_{\varphi}}(e^{A}(|\partial  \varphi|_{\omega_0}^2 + K))\\ &= \Delta_{\omega_{\varphi}}(e^{A})(|\partial  \varphi|_{\omega_0}^2 + K) + e^{A} \Delta_{\omega_{ \varphi}}(|\partial \varphi|_{\omega_0}^2) + 2g^{i \bar i}_{\varphi}e^{A}\text{Re}(A_i(|\partial \varphi|_{\omega_0}^2)_{\bar i}).
\end{aligned}
\end{equation}
Let us estimate the three terms involved in the above equation separately
\begin{equation}
    \begin{aligned}
    &\Delta_{\omega_{\varphi}}(e^{A})\\
    &= e^{A}  g^{i \bar i}_{\varphi}|A_i|^2 + e^{A}(-\Delta_{\omega_{\varphi}}(F+\lambda \varphi) +  \varphi \Delta_{\omega_{\varphi}} \varphi + Q \delta' \Delta_{\omega_{\varphi}}\psi) + e^{A} g^{i \bar i}_{\varphi} | \varphi_i|^2\\
    &= e^{A}(\underline{R} - \lambda n + (\lambda-\varphi) \text{tr}_{\omega_{\varphi}} \omega  + \text{tr}_{\omega_{\varphi}}\eta + n  \varphi + Q \delta' \Delta_{\omega_{\varphi}}\psi) + e^{A}g^{i \bar i}_{\varphi} | \varphi_i|^2\\
    &+e^{A}  g^{i \bar i}_{\varphi}|A_i|^2,
    \end{aligned}
\end{equation}
and
\begin{equation}
    \begin{aligned}
    \Delta_{\omega_\varphi}(|\partial  \varphi|_{\omega_0}^2) =  g^{i \bar i}_{\varphi}R_{\alpha \beta i \bar i}(g_0)  \varphi_{\alpha} \varphi_{\bar \beta} +  g^{i \bar i}_{\varphi}| \varphi_{\alpha i}|^2 +  g^{i \bar i}_{\varphi}  | \varphi_{\alpha \bar i}|^2 +  g^{i \bar i}_{\varphi}(\varphi_{\alpha i \bar i}   \varphi_{\bar \alpha} +   \varphi_{\bar \alpha i \bar i} \varphi_{\alpha}).
    \end{aligned}
\end{equation}
Differentiating
\begin{equation}
    \mathrm{log} \frac{\omega_{\varphi}^n}{\omega_0^n} = F,
\end{equation}
we get
\begin{equation}
     g^{ i \bar i}_{\varphi}(  g_{i \bar i, \alpha} + \varphi_{i \bar i \alpha} ) = F_{\alpha} \text{ and }  g^{ i \bar i}_{\varphi}( g_{i \bar i, \bar \alpha} + \varphi_{i \bar i \bar \alpha} ) = F_{\bar \alpha}.
\end{equation}
So 
\begin{equation}
    \begin{aligned}
     &\Delta_{\omega_\varphi}(|\partial \varphi|_{\omega_0}^2) \\
     &= g^{i \bar i}_{\varphi}R_{\alpha \bar \beta i \bar i}(g_0) \varphi_{\alpha}  \varphi_{\bar \beta} + g^{i \bar i}_{\varphi}| \varphi_{\alpha i}|^2 + g^{i \bar i}_{\varphi} |\varphi_{\alpha \bar i}|^2 +  F_{\alpha}  \varphi_{\bar \alpha} + F_{\bar \alpha} \varphi_{\alpha} - 2 g^{i \bar i}_{\varphi}\text{Re}(g_{i \bar i, \alpha}  \varphi_{\bar \alpha})
     \\
     & \geq - C_{1}|\partial  \varphi|_{\omega_0}^2 \mathrm{tr}_{\omega_{\varphi}}\omega_{0} +  g^{i \bar i}_{\varphi}| \varphi_{\alpha i}|^2 +  g^{i \bar i}_{\varphi} | \varphi_{\alpha \bar i}|^2 - 2\text{Re}(A_{\alpha} \varphi_{\bar \alpha})  -2 (\lambda -  \varphi) |\partial  \varphi|_{\omega_0}^2 \\
     &+2 Q\delta'\text{Re}(\psi_{\alpha} \varphi_{\bar \alpha}) - 2 g^{i \bar i}_{\varphi}\text{Re}(g_{i \bar i, \alpha}  \varphi_{\bar \alpha}).
    \end{aligned}
\end{equation}
where $C_{1}$ depends on a lower bound on the bisectional curvature of $\omega_0$. Then we have
\begin{equation}
    \begin{aligned}
    & e^{-A} \Delta_{\omega_\varphi}(e^A(|\partial \varphi|_{\omega_0}^2 + K)) \\
    & \geq (\underline{R} - \lambda n + (\lambda-\varphi) \text{tr}_{\omega_{\varphi}} \omega + Q\delta' \Delta_{\omega_{\varphi}}\psi + \text{tr}_{\omega_{ \varphi}}\eta + n \varphi)(|\partial  \varphi|_{\omega_0}^2 + K)\\
     &+ g^{i \bar i}_{\varphi}|A_i|^2(|\partial  \varphi|_{\omega_0}^2 + K)
    +|\partial  \varphi|^2_{\omega_{\varphi}}(|\partial  \varphi|_{\omega_0}^2 + K) - C_{1}|\partial \varphi|_{\omega_0}^2 \mathrm{tr}_{\omega_{\varphi}}\omega_{0} + g^{i \bar i}_{\varphi}| \varphi_{\alpha i}|^2 \\
    &+ g^{i \bar i}_{\varphi} |\varphi_{\alpha \bar i}|^2 + (-2 \lambda + 2  \varphi)|\partial \varphi|_{\omega_0}^2 - 2\text{Re}(A_{\alpha} \varphi_{\bar \alpha}) + 2g^{i \bar i}_{\varphi}\text{Re}(A_i( \varphi_{\alpha \bar i} \varphi_{\bar \alpha} +  \varphi_{\bar \alpha \bar i}  \varphi_{\alpha})) \\
    &+ 2Q \delta' \text{Re}(\psi_{\alpha} \varphi_{\bar \alpha}) - 2 g^{i \bar i}_{\varphi}\text{Re}( g_{i \bar i, \alpha} \varphi_{\bar \alpha}).
    \end{aligned}
\end{equation}
Notice the following complete square in the above sum
\begin{equation}
    \begin{aligned}
    &g^{i \bar i}_{\varphi}|\varphi_{i \alpha} + A_i \varphi_{\alpha}|^2\\
    &=g^{i \bar i}_{\varphi}|\varphi_{i \alpha}|^2 + 2g^{i \bar i}_{\varphi} \text{Re}(A_i \varphi_{\alpha}\varphi_{\bar \alpha \bar i}) +g^{i \bar i}_{\varphi}|A_i|^2|\partial  \varphi|_{\omega_0}^2.
    \end{aligned}
\end{equation}
Also observe that
\begin{equation}
    \begin{aligned}
    &-A_{\alpha}\varphi_{\bar \alpha} + g^{i \bar i}_{\varphi}A_i  \varphi_{ \alpha \bar i} \varphi_{\bar \alpha}\\
    &= g^{i \bar i}_{\varphi}(A_{i} \varphi_{\alpha \bar i} \varphi_{\bar \alpha} -(g_{\varphi})_{\alpha \bar i} A_i \varphi_{\bar \alpha})\\
    &= -g^{i \bar i}_{\varphi}  g_{\alpha \bar i}A_i \varphi_{\bar \alpha}.
    \end{aligned}
\end{equation}
In summary 
\begin{equation}
    \begin{aligned}
    &\Delta_{\varphi}(e^A(|\partial  \varphi|_{\omega_0}^2 + K))e^{-A}\\ &\geq K g^{i \bar i}_{\varphi}|A_i|^2 + g^{i \bar i}_{\varphi}|\varphi_i|^2(|\partial \varphi|_{\omega_0}^2 + K)\\ &+(\text{tr}_{\omega_{\varphi}}(\lambda -  \varphi)  \omega +  Q \delta' \Delta_{\omega_{\varphi}}\psi + \text{tr}_{\omega_{\varphi}}\eta)(|\partial  \varphi|_{\omega_0}^2 + K)\\
    & + (\underline{R} - \lambda n + n \varphi)(|\partial \varphi|_{\omega_0}^2 + K) - C_{1}|\partial \varphi|_{\omega_0}^2\mathrm{tr}_{\omega_{\varphi}}\omega_{0} + g^{i \bar i}_{\varphi}|\varphi_{\alpha \bar i}|^{2}\\ &+(-2\lambda + 2\varphi)|\partial \varphi|_{\omega_0}^2 
     - 2g^{i \bar i}_{\varphi}\text{Re}( g_{\alpha \bar i} A_i \varphi_{\bar \alpha}) + 2Q\delta'\text{Re}(\psi_{\alpha} \varphi_{\bar \alpha})-2\text{Re}( g^{i \bar i}_{\varphi}  g _{i \bar i, \alpha}  \varphi_{\bar \alpha}).
    \end{aligned}
\end{equation}
Furthermore, by Young's inequality
\begin{equation}
        2\text{Re}(g^{i \bar i}_{\varphi} g_{i \bar i, \alpha} \varphi_{\bar \alpha}) \leq C_{2}(\mathrm{tr}_{\omega_{\varphi}}\omega_{0}  +  |\partial \varphi|_{\omega_0}^2 \mathrm{tr}_{\omega_{\varphi}}\omega_{0}),
\end{equation}
and
\begin{equation}
   2\text{Re}( g^{i \bar i}_{\varphi} g_{\alpha \bar i}A_i \varphi_{\bar \alpha}) \leq C_{3} (g^{i \bar i}_{\varphi} |A_{i}|^2 +  |\partial \varphi|_{\omega_0}^2 \mathrm{tr}_{\omega_{\varphi}}\omega_{0}).
\end{equation}
Let $\lambda = ( \frac{1}{\delta'}+ \frac{1}{C_{\delta'}})(\|\varphi\|_{C^{0}(M)} + 10 + C_{1} + C_{2} + C_{3})$, $Q = \lambda - \|\varphi\|_{C^{0}(M)} > \max \{ \frac{10 + C_{1} + C_{2}+C_{3}}{C_{\delta'}}, \frac{2}{\delta'} \} $, $K = C_{2} + C_{3}$, so that $C_{\delta'} Q \geq 10 + C_{1} + C_{2} + C_{3}$, $Q \delta' > 1$ and recall that $\psi$ is chosen such that
\begin{equation}
    \omega + \delta' \sqrt{-1} \partial \bar \partial \psi \geq \eta + \delta' \sqrt{-1} \partial \bar \partial \psi >  C_{\delta'} \omega_0.
\end{equation}
So
\begin{equation}
\begin{aligned}
    &(\text{tr}_{\omega_{\varphi}}(\lambda - \varphi)  \omega + Q \delta' \Delta_{\omega_{\varphi}}  \psi + \text{tr}_{\omega_{\varphi}}\eta)(|\partial \varphi|_{\omega_0}^2 + K)  - C_{1}|\partial \varphi|_{\omega_0}^ 2 \mathrm{tr}_{\omega_{\varphi}}\omega_{0} \\
    &  - 2\text{Re}(g^{i \bar i}_{\varphi} g_{i \bar i, \alpha} \varphi_{\bar \alpha}) + K g^{i \bar i}_{\varphi}|A_i|^2  + 2\text{Re} (g^{i \bar i}_{\varphi} g_{\alpha \bar i} A_{i} \varphi_{\bar \alpha})\\
    &\geq Q(\text{tr}_{\omega_{\varphi}} \omega + \delta' \Delta_{\omega_{\varphi}}  \psi) (|\partial \varphi|_{\omega_0}^2 + K) - (C_{1} + C_{2} + C_{3})|\partial \varphi|_{\omega_0}^2\mathrm{tr}_{\omega_{\varphi}}\omega_{0}\\
    &- C_{2}\mathrm{tr}_{\omega_{\varphi}}\omega_{0}  + (K - C_{3}) g^{i \bar i}_{\varphi}|A_i|^2\\
    &\geq QC_{\delta'}(|\partial \varphi|_{\omega_0}^2 + K)\mathrm{tr}_{\omega_{\varphi}}\omega_{0} - (C_{1} + C_{2} + C_{3})|\partial \varphi|_{\omega_0}^2 \mathrm{tr}_{\omega_{\varphi}}\omega_{0} - C_{2}\mathrm{tr}_{\omega_{\varphi}}\omega_{0}\\
    & \geq 10 |\partial \varphi|_{\omega_0}^2 \mathrm{tr}_{\omega_{\varphi}}\omega_{0}.
\end{aligned}
\end{equation}
Recall that $\underline{R} \to -n$, so
\begin{equation}
    (\underline{R} - \lambda n + n \varphi)(|\partial \varphi|_{\omega_0}^2 + K) \geq -C|\partial \varphi|_{\omega_0}^2  - C,
\end{equation}
and
\begin{equation}
    (-2\lambda + 2\varphi)|\partial \varphi|_{\omega_0}^2 \geq -C |\partial \varphi|_{\omega_0}^2
\end{equation}
where the constants depend on the $0$-th order estimate on $\varphi$. Also as an elementary consequence of the equation (see \cite{1712.06697} page 12)
\begin{equation}
 \frac{\omega_{\varphi}^n}{\omega_0^n} = e^{F}
\end{equation}
yields 
\begin{equation}
    g^{i \bar i}_{\varphi}|\varphi_i|^2 |\partial \varphi|_{\omega_0}^2 + |\partial \varphi|_{\omega_0}^2 \mathrm{tr}_{\omega_{\varphi}}\omega_{0}\geq C |\partial \varphi|_{\omega_0}^{2 + \frac{2}{n}}e^{\frac{-F}{n}}.
\end{equation}
Recall that $\psi = \text{log} |s|_{h}^{2}$, where $s$ is a holomorphic section vanishing on $E$. Then we get
\begin{equation}
    g^{i \bar i}_{\varphi}|\varphi_i|^2 |\partial \varphi|_{\omega_0}^2 + |\partial \varphi|_{\omega_0}^2 \mathrm{tr}_{\omega_{\varphi}}\omega_{0} \geq  C |\partial \varphi|_{\omega_0}^{2 + \frac{2}{n}}e^{\frac{-F + Q\delta'\psi}{n}}.
\end{equation}
We arrive at
\begin{equation}
\begin{aligned}
    &\Delta_{\omega_{\varphi}}(e^{A}(|\partial \varphi|_{\omega_0}^2 + K)) \\
    &\geq e^{A}(g^{i \bar i}_{\varphi}|\varphi_{i}|^2 |\partial \varphi|_{\omega_0}^2 + |\partial \varphi|_{\omega_0}^2\mathrm{tr}_{\omega_{\varphi}}\omega_{0}- C|\partial \varphi|_{\omega_0}^2 - C + 2Q \delta' \mathrm{Re}(\psi_{\alpha} \varphi_{\bar \alpha})) \\
    &\geq e^{A}(Ce^{\frac{-F + Q \delta' \psi}{n}}(|\partial \varphi|_{\omega_0}^2)^{1 + \frac{1}{n}} -  C |\partial \varphi|_{\omega_0}^2 - C + 2Q \delta' \mathrm{Re}(\psi_{\alpha} \varphi_{\bar \alpha}))\\
    &\geq C(e^{Q \delta' \psi}|\partial \varphi|_{\omega_0}^2)^{1 + \frac{1}{n}} -  C e^{Q \delta' \psi}|\partial \varphi|_{\omega_0}^2 - C e^{Q\delta'\psi} +  2e^{A}Q\delta'\mathrm{Re}(\psi_{\alpha} \varphi_{\bar \alpha}) \\
\end{aligned}
\end{equation}
where we used the fact that $\|F\|_{C^{0}(M)}$ and $\|\varphi\|_{C^{0}(M)}$ are bounded. For the last term we have
\begin{equation}
 \psi_{\alpha} = \frac{\partial_{\alpha}(|s|_{h}^{2})}{|s|_{h}^{2}},
\end{equation}
So 
\begin{equation}
\begin{aligned}
  &2e^{A}Q\delta'\mathrm{Re}(\psi_{\alpha} \varphi_{\alpha})\\
   &\geq -C e^{2\psi} |\partial \psi|_{\omega_0}^{2} - Ce^{(2Q\delta' - 2)\psi}|\partial \varphi|^{2}_{\omega_0}\\
   &\geq -C - Ce^{Q\delta'\psi}|\partial \varphi|^2_{\omega_0}.
\end{aligned}
\end{equation}
where we used the fact that  $Q$ is chosen such that $2Q\delta' - 2 > Q\delta'$ so $e^{(2Q\delta' - 2) \psi} \leq C e^{Q \delta' \psi}$. Finally we reach
\begin{equation}
\begin{aligned}
    &\Delta_{\omega_{\varphi}}(e^{A}(|\partial \varphi|_{\omega_0}^2 + K)) \\
    &\geq C (e^{Q \delta' \psi}|\partial \varphi|_{\omega_0}^2)^{1 + \frac{1}{n}} -  C e^{Q \delta' \psi}|\partial \varphi|_{\omega_0}^2 - C e^{Q \delta' \psi}  -C.\\
\end{aligned}
\end{equation}
At a maximum $p$ of $e^{A}(|\partial \varphi|_{\omega_0}^2 + K)$, which is away from $E$ because $A = -\infty$ on $E$, we have
\begin{equation}
\begin{aligned}
      &C (e^{Q\delta'\psi(p)}|\partial \varphi|_{\omega_0}^2(p))^{1 + \frac{1}{n}} - C e^{Q\delta'\psi(p)}|\partial \varphi|_{\omega_0}^2(p)\\
      &\leq C e^{Q \delta' \psi(p)} + C\\
      &\leq C
\end{aligned}
\end{equation}
where for the last inequality we used the fact that $e^{P\psi}$ is bounded from above for any constant $P > 0$, and our choice that $Q \delta' > 1$. Thus, we conclude that
$e^{Q \delta' \psi} |\partial \varphi|_{\omega_0}^2(p)$ is bounded from above. As a consequence, $e^{A}(|\partial \varphi|_{\omega_0}^2 + K)(p)$ is also bounded from above. This concludes our proof.
\end{proof}

\section{Degenerate \texorpdfstring{$L^p$}{TEXT} bound on \texorpdfstring{$\mathrm{tr}_{\omega_0} \label{fourthsection} \omega_{\varphi_{\varepsilon}}$ and $C^{\infty}_{\mathrm{loc}}(M \backslash E)$}{TEXT} bound on \texorpdfstring{$\varphi_{\varepsilon}$}{TEXT}}

In this section we will first establish a degenerate version of Theorem 3.1 in \cite{1712.06697}.

\begin{Th}\label{lpbound}
For each $p > 0$, there exists a constant $\gamma(p) > 0$ such that 
\begin{equation}
    \int_{M} e^{\gamma(p) \psi} (\mathrm{tr}_{\omega_0} \omega_{\varphi_{\varepsilon}})^p \omega_0^n \leq C(p, (M , \omega_0), \|\varphi_{\varepsilon}\|_{C^{0}(M)}, \|F_{\varepsilon}\|_{C^{0}(M)}).
\end{equation}
\end{Th}
For simplicity, let $B_{\varepsilon} = F_{\varepsilon}+\lambda \varphi_{\varepsilon} - \lambda \delta' \psi$, where $\delta' < \min\{\frac{1}{2},\sigma \}$ and $\sigma$ is specified in Lemma \ref{Tsuji}. Moreover, let $\beta$ be a constant greater than $1$. We start by computing at a point away from $E$ in a normal holomorphic coordinate neighborhood for $g_0$ such that $g_{\varphi}$ is diagonal.
\begin{equation}
    \begin{aligned}
    &\Delta_{\omega_{\varphi}}(e^{-\beta B}(\text{tr}_{\omega_0} \omega_{\varphi}))\\
    &= \Delta_{\omega_{\varphi}}(e^{-\beta B})\text{tr}_{\omega_0} \omega_{\varphi} + e^{-\beta B} \Delta_{\omega_{\varphi}}(\text{tr}_{\omega_0} \omega_{\varphi})-2\beta \mathrm{Re}(e^{-\beta B}g^{i \bar i}_{\varphi}(B_i \partial_{\bar i} \text{tr}_{\omega_0} \omega_{\varphi})),\\
    \end{aligned}
\end{equation}
and
\begin{equation}
    \begin{aligned}
    \Delta_{\omega_{\varphi}}e^{-\beta B} &= (\beta^2 g^{i \bar i}_{\varphi} |B_i|^2 + \beta (\underline{R} -\lambda n + \lambda \text{tr}_{\omega_{\varphi}}(\omega + \delta ' \sqrt{-1} \partial \bar \partial \psi) + \mathrm{tr}_{\omega_{\varphi}}\eta))e^{-\beta B}\\
    & \geq \beta^2 g^{i \bar i}_{\varphi}|B_i|^2e^{-\beta B} + \beta e^{-\beta B}(\underline{R} - \lambda n) +   C_{\delta'} \lambda \beta e^{-\beta B}\text{tr}_{\omega_{\varphi}}\omega_0.
    \end{aligned}
\end{equation}
Now we estimate $\Delta_{\omega_{\varphi}}\text{tr}_{\omega_0} \omega_{\varphi}$ by first recalling that from the calculation in Yau \cite{yau1978ricci} (see also \cite{MR3186384} Lemma 3.7 for an exposition) we have
\begin{equation}
    \Delta_{\omega_{\varphi}} \log \text{tr}_{\omega_0} \omega_{\varphi} \geq - C_{1} \text{tr}_{\omega_{\varphi}} \omega_0 - \frac{g_0^{i \bar j}R_{i \bar j}(g_{\varphi})}{\text{tr}_{\omega_0} \omega_{\varphi}}
\end{equation}
where $C_{1}$ only depends on a lower bound for the bisectional curvature of $g_0$. Recall that 
\begin{equation}
  \log \frac{ \omega_{\varphi}^n}{\omega_0^n} = F.
\end{equation}
Then we have
\begin{equation}
    - \mathrm{tr}_{\omega_0} \mathrm{Ric}(\omega_{\varphi}) = \mathrm{tr}_{\omega_0} \eta + \Delta_{\omega_0} F.
\end{equation}
Thus
\begin{equation}
    \Delta_{\omega_{\varphi}} \text{log} \text{tr}_{\omega_0} \omega_{\varphi} \geq - C_{1} \text{tr}_{\omega_{\varphi}} \omega_0 + \frac{\Delta_{\omega_0} F}{\text{tr}_{\omega_0} \omega_{\varphi}}  + \frac{\text{tr}_{\omega_0} \eta}{\text{tr}_{\omega_0} \omega_{\varphi}},
\end{equation}
and
\begin{equation}
\Delta_{\omega_{\varphi}} \text{tr}_{\omega_0}\omega_{\varphi} \geq -C_{1} \text{tr}_{\omega_{\varphi}} \omega_0 \text{tr}_{\omega_0} \omega_{\varphi}+ \frac{|\partial \text{tr}_{\omega_0}\omega_{\varphi}|_{\omega_{\varphi}}^2}{\text{tr}_{\omega_0} \omega_{\varphi}} + \Delta_{\omega_0} F + \text{tr}_{\omega_0} \eta.
\end{equation}
Then we are able to conclude that
\begin{equation}
    \begin{aligned}
    &e^{\beta B} \Delta_{\omega_{\varphi}}(e^{-\beta B}(\text{tr}_{\omega_0} \omega_{\varphi}))  \\
    &\geq (\beta^2 g^{i \bar i}_{\varphi} |B_i|^2 + \beta(\underline{R} - \lambda n))\text{tr}_{\omega_0} \omega_{\varphi}  + (C_{\delta'} \lambda \beta  - C_{1}) \text{tr}_{\omega_{\varphi}} \omega_0 \text{tr}_{\omega_0} \omega_{\varphi}\\
    &+ \frac{|\partial \text{tr}_{\omega_0}\omega_{\varphi}|_{\omega_{\varphi}}^2}{\text{tr}_{\omega_0} \omega_{\varphi}} + \Delta_{\omega_0} F
    + \text{tr}_{\omega_0} \eta
    -2\beta \text{Re}(g^{i \bar i}_{\varphi}(B_i \partial_{\bar i} \text{tr}_{\omega_0} \omega_{\varphi}))\\
    &\geq  \beta(\underline{R} - \lambda n)\text{tr}_{\omega_0} \omega_{\varphi} +  (C_{\delta'} \lambda \beta  - C_{1}) \text{tr}_{\omega_{\varphi}} \omega_0 \text{tr}_{\omega_0}\omega_{\varphi} + \Delta_{\omega_0} F + \text{tr}_{\omega_{0}} \eta
    \end{aligned}
\end{equation}
where we dropped  the terms
\begin{equation}
    \begin{aligned}
    &\beta^2 g^{i \bar i}_{\varphi}|B_i|^2 \text{tr}_{\omega_0} \omega_{\varphi} - 2\beta \text{Re}(g^{i \bar i}_{\varphi}(B_i \partial_{\bar i}\text{tr}_{\omega_0} \omega_{\varphi})) + \frac{|\partial \text{tr}_{\omega_0}\omega_{\varphi}|_{\omega_{\varphi}}^2}{\text{tr}_{\omega_0} \omega_{\varphi}}\\
    &= \text{tr}_{\omega_0} \omega_{\varphi} g^{i \bar i}_{\varphi}(\beta^2 |B_i|^2  - 2 \beta \text{Re}(\frac{B_i \partial_{\bar i} \text{tr}_{\omega_0} \omega_{\varphi}}{\text{tr}_{\omega_0}\omega_{\varphi}}) + \frac{|\partial_{\bar i} \text{tr}_{\omega_0}\omega_{\varphi}|^2}{(\text{tr}_{\omega_0} \omega_{\varphi})^2}) \geq 0.\\
    \end{aligned}
\end{equation}
Use
\begin{equation}
    \begin{aligned}
    \text{tr}_{\omega_{\varphi}} \omega_0 \text{tr}_{\omega_0} \omega_{\varphi} \geq e^{-\frac{F}{n-1}}(\text{tr}_{\omega_0} \omega_{\varphi})^{1+\frac{1}{n-1}},
    \end{aligned}
\end{equation}
and choose $\lambda = \max( \frac{2C_{1} + 2}{C_{\delta'}},\frac{q}{2 \delta'} + 1, \frac{1}{2 \delta'} + 1)$, where $q$ is specified in Theorem \ref{gradientbound222}, so that $C_{\delta'} \lambda \beta \geq (2C_{1} + 2) \beta > 2C_{1} + 1$. We get
\begin{equation} \label{4.12}
    \begin{aligned}
    &\Delta_{\omega_{\varphi}}(e^{-\beta B}(\text{tr}_{\omega_0} \omega_{\varphi})) \\ &\geq \beta(\underline{R} - \lambda n)\text{tr}_{\omega_0} \omega_{\varphi}e^{-\beta B} + (C_{\delta'} \lambda \beta - C_{1})e^{-\frac{F}{n-1} - \beta B} (\text{tr}_{\omega_0} \omega_{\varphi})^{1+\frac{1}{n-1}}\\
    &+(\Delta_{\omega_0} F +  \text{tr}_{\omega_0} \eta)e^{-\beta B}\\
    &\geq \beta(\underline{R} - \lambda n)\text{tr}_{\omega_0} \omega_{\varphi}e^{-\beta B} + (C_{\delta'} \lambda \beta - C_{1})e^{-\frac{F}{n-1} - \beta B} (\text{tr}_{\omega_0} \omega_{\varphi})^{1+\frac{1}{n-1}}\\
    &+ \Delta_{\omega_0} F e^{- \beta B}
     \end{aligned}
\end{equation}
where for the last line we took advantage of the fact that $\eta \geq 0$. Let $u = e^{- \beta B} \text{tr}_{\omega_0} \omega_{\varphi}$, for any $p \geq 0$, we have
\begin{equation} \label{4.13}
\begin{aligned}
    &\frac{1}{2p+1} \Delta_{\omega_{\varphi}}u^{2p + 1} \\
    &= 2pu^{2p-2}e^{-\beta B}(\text{tr}_{\omega_0}\omega_{\varphi})|\partial u|_{\omega_{\varphi}}^2 + u^{2p} \Delta_{\omega_\varphi} u 
    \geq 2pu^{2p-2}|\partial u|_{\omega_0}^2e^{-\beta B} + u^{2p} \Delta_{\omega_{\varphi}}u.
\end{aligned}
\end{equation}
Integrating both sides of (\ref{4.13}) with respect to $\omega_{\varphi}^n = e^{F}\omega_0^n$ and use (\ref{4.12}) we get
\begin{equation} \label{firstintegration}
    \begin{aligned}
    &\int_{M} 2pu^{2p-2}|\partial u|_{\omega_0}^2e^{-\beta B + F} \omega_0^n +\int_{M} \Delta_{\omega_0} F e^{-\beta B + F} u^{2p}\omega_0^n \\
    &+ \int_{M} u^{2p}e^{-\beta B + \frac{n-2}{n-1}F}(C_{\delta'} \lambda \beta - C_{1})(\text{tr}_{\omega_0}\omega_{\varphi})^{1+\frac{1}{n-1}} \omega_0^n \\
    &\leq \int_{M} \beta(\lambda n - \underline{R}) \text{tr}_{\omega_0}\omega_{\varphi}e^{-\beta B + F}u^{2p} \omega_0^n.
    \end{aligned}
\end{equation}
Let us denote $\tilde \varphi = \varphi - \delta' \psi$, and to handle the term involving $\Delta_{\omega_0} F$, we apply integration by parts
\begin{equation}\label{integration}
\begin{aligned}
        &\int_{M} e^{-\beta B + F} u^{2p} \Delta_{\omega_0} F  \omega_0^n  = \int_{M} (\beta-1) e^{(1-\beta)F - \beta \lambda \tilde \varphi} u^{2p} |\partial F|_{\omega_0}^2 \omega_0^n\\
        &+\int_{M} \beta \lambda e^{(1-\beta)F - \lambda \beta \tilde \varphi}u^{2p}\langle \partial  \tilde \varphi, \partial F \rangle_{\omega_0} \omega_0^n - \int_{M} 2p e^{(1-\beta)F - \beta \lambda \tilde \varphi}u^{2p-1}\langle \partial u, \partial F \rangle_{\omega_0} \omega_0^n,
\end{aligned}
\end{equation}
where $\langle \partial \tilde \varphi,  \partial F \rangle_{\omega_0} = g_{0}^{i \bar j} \tilde \varphi_{i} F_{\bar j}$, and $\langle \partial u,  \partial F \rangle_{\omega_0} = g_{0}^{i \bar j} u_{i} F_{\bar j}$ in coordinates. Notice that by Young's inequality
\begin{equation}
    \begin{aligned}
    |u^{2p-1}\langle \partial  u,  \partial F  \rangle_{\omega_0}| \leq \frac{1}{2}u^{2p}|\partial F|_{\omega_0}^2 + \frac{1}{2} u^{2p-2} |\partial u|_{\omega_0}^2,
    \end{aligned}
\end{equation}
and
\begin{equation}
    \begin{aligned}
    &|\beta \lambda e^{(1-\beta)F - \lambda \beta \tilde \varphi}u^{2p}\langle  \partial \tilde  \varphi,  \partial F \rangle_{\omega_0} |\\
    & \leq |\beta \lambda e^{(1-\beta)F - \lambda \beta \tilde \varphi}u^{2p}\langle \partial  \varphi,  \partial F \rangle_{\omega_0}| + | \delta' \beta \lambda e^{(1-\beta)F - \lambda \beta \tilde \varphi}u^{2p}\langle \partial \psi,  \partial F \rangle_{\omega_0} |\\
    & \leq \frac{(\beta - 1)}{2}u^{2p}|\partial F|_{\omega_0}^2 e^{(1-\beta)F -  \lambda \beta \tilde \varphi } + |\partial \varphi|_{\omega_0}^2 \frac{\beta^2 \lambda^2}{(\beta -1)} u^{2p}e^{(1-\beta)F - \lambda \beta \tilde \varphi} \\
    &+ \delta'^{2} |\partial \psi|_{\omega_0}^2 \frac{\beta^2 \lambda^2}{(\beta -1)} u^{2p}e^{(1-\beta)F - \lambda \beta \tilde \varphi}.
    \end{aligned}
\end{equation}
Then substitute this back into (\ref{integration}), we get
\begin{equation} \label{furtherintegration}
    \begin{aligned}
    &\int_{M}e^{(1- \beta)F - \beta \lambda \tilde \varphi}u^{2p} \Delta_{\omega_0} F \omega_0^n \\
      &\geq -  \int_{M}|\partial \psi|_{\omega_0}^2 \frac{\delta'^{2} \beta^2 \lambda^2}{(\beta - 1)} e^{(1-\beta)F - \lambda \beta \tilde \varphi} u^{2p} \omega_0^n - \int_{M} pe^{(1-\beta)F - \lambda \beta \tilde \varphi} u^{2p-2}|\partial u|_{\omega_0}^2 \omega_0^n\\
    &+ \int_{M} (\frac{\beta - 1}{2} - p)e^{(1-\beta) F - \beta \lambda \tilde \varphi} u^{2p}|\partial F|_{\omega_0}^2 \omega_0^n \\
    &- \int_{M}|\partial \varphi|_{\omega_0}^2 \frac{\beta^2 \lambda^2}{(\beta - 1)} e^{(1-\beta)F - \lambda \beta \tilde \varphi} u^{2p} \omega_0^n.
    \end{aligned}
\end{equation}
Substituting (\ref{furtherintegration}) back into (\ref{firstintegration}), we conclude

\begin{equation}
    \begin{aligned}
            &\int_{M} pu^{2p-2}|\partial u|_{\omega_0}^2e^{(1-\beta) F - \lambda \beta \tilde \varphi} \omega_0^n +\int_{M} (\frac{\beta - 1}{2} - p)e^{(1-\beta) F - \beta \lambda \tilde \varphi} u^{2p}|\partial F|_{\omega_0}^2 \omega_0^n  \\
            &+ \int_{M} u^{2p}e^{-(\beta - \frac{n-2}{n-1} ) F- \beta \lambda \tilde \varphi}(C_{\delta'} \beta \lambda  - C_{1})(\text{tr}_{\omega_0}\omega_{\varphi})^{1+\frac{1}{n-1}} \omega_0^n  \\
    &\leq \int_{M}|\partial \varphi|_{\omega_0}^2 \frac{\beta^2 \lambda^2}{(\beta - 1)} e^{(1-\beta)F - \lambda \beta   \tilde \varphi} u^{2p} \omega_0^n  + \int_{M}|\partial \psi|_{\omega_0}^2 \frac{\delta'^2\beta^2 \lambda^2}{(\beta - 1)}e^{(1-\beta)F - \lambda \beta   \tilde \varphi} u^{2p} \omega_0^n\\
    &+\int_{M} \beta(\lambda n - \underline{R}) \text{tr}_{\omega_0}\omega_{\varphi}e^{(1-\beta)F - \beta \lambda \tilde \varphi}u^{2p} \omega_0^n.
    \end{aligned}
\end{equation}
Now choose $\beta > 2p+1$, and recall that $\lambda = \max( \frac{2C_{1} + 2}{C_{\delta'}},\frac{q}{2 \delta'} + 1, \frac{1}{2 \delta'} + 1)$, where $q$ is specified in Theorem \ref{gradientbound222}, so that $C_{\delta'} \lambda \beta \geq (2C_{1} + 2) \beta > 2C_{1} + 1$. Because $\text{tr}_{\omega_0} \omega_{\varphi}  \geq   n e^{\frac{F}{n}}$ by the arithmetic-geometric means inequality, we have

\begin{equation} \label{maintool2}
    \begin{aligned}
    & \int_{M} u^{2p}e^{-(\beta - \frac{n-2}{n-1} ) F- \beta \lambda \tilde \varphi}(\text{tr}_{\omega_0}\omega_{\varphi})^{1+\frac{1}{n-1}} \omega_0^n  \\
    &\leq C \int_{M} \text{tr}_{\omega_0}\omega_{\varphi}e^{(1-\beta)F - \beta \lambda \tilde \varphi}u^{2p} \omega_0^n + C \int_{M} u^{2p} e^{(1-\beta) F - \beta \lambda \tilde \varphi} \omega_0^n\\
    &+C \int_{M}|\partial \varphi|_{\omega_0}^2 \frac{\beta^2 }{\beta - 1} e^{(1-\beta)F - \lambda \beta \tilde \varphi} u^{2p} \omega_0^n + C \int_{M}|\partial \psi|_{\omega_0}^2 \frac{\delta'^{2} \beta^2 }{\beta - 1} e^{(1-\beta)F - \lambda \beta \tilde \varphi} u^{2p} \omega_0^n.
    \end{aligned}
\end{equation}
For $p = 0$, $\beta = 2$  we have
\begin{equation}
    \begin{aligned}
    & \int_{M} e^{2\lambda \delta' \psi}(\text{tr}_{\omega_0} \omega_{\varphi})^{1+\frac{1}{n-1}} \omega_0^n \\ 
    &\leq C \int_{M} e^{ - (2 - \frac{n-2}{n-1}) F- 2 \lambda \tilde \varphi}(\text{tr}_{\omega_0}\omega_{\varphi})^{1+\frac{1}{n-1}} \omega_0^n \\
    &\leq C \int_{M} \text{tr}_{\omega_0}\omega_{\varphi}e^{ 2 \lambda \delta' \psi} \omega_0^n + C \int_{M} e^{2\lambda \delta' \psi } \omega_0^n\\
    &+C\int_{M}|\partial \varphi|_{\omega_0}^2 e^{-2\lambda \varphi + 2\lambda \delta' \psi} \omega_0^n +C\int_{M}|\partial \psi|_{\omega_0}^2 e^{-2\lambda \varphi + 2\lambda \delta' \psi} \omega_0^n 
    \\
    &\leq C \int_{M} \text{tr}_{\omega_0} \omega_{\varphi} \omega_0^n + C + C \int_{M} |\partial \varphi|_{\omega_0}^2 e^{2\lambda \delta' \psi} \omega_0^n + C \int_{M} |\partial \psi|_{\omega_0}^2 e^{2\lambda \delta' \psi} \omega_0^n,
\end{aligned}
 \end{equation}
where we used the fact that $\varphi$ and $F$ are bounded. Also notice that 
\begin{equation}
\int_{M} \text{tr}_{\omega_0} \omega_{\varphi} \omega_0^n = \int_{M} \text{tr}_{\omega_0} \omega \omega_0^n + \int_{M} \Delta_{\omega_0} \varphi  \omega_0^n \leq C.
\end{equation}
For the last two terms
\begin{equation}
\begin{aligned}
\int_{M} |\partial \varphi|_{\omega_0}^2 e^{2\lambda \delta' \psi} \omega_0^n &\leq C \int_{M} |s|_{h}^{4\lambda \delta' - 2q}\omega_0^n \leq C,
\end{aligned}
\end{equation}
and
\begin{equation}
\begin{aligned}
\int_{M} |\partial \psi|_{\omega_0}^2 e^{2\lambda \delta' \psi} \omega_0^n &\leq C \int_{M} |s|_{h}^{4\lambda \delta' - 2}\omega_0^n \leq C.
\end{aligned}
\end{equation}
Recall that we choose $\lambda > \max\{\frac{q}{2\delta'}, \frac{1}{2\delta'}\}$, so we can bound the last two terms since $|s|_{h}^2$ is bounded. As a result, we obtained a bound for $\int_{M} e^{2 \lambda \delta' \psi}(\text{tr}_{\omega_0} \omega_{\varphi})^{\frac{n}{n-1}}\omega_0^n$. Now it suffices to show that for each given $p > 0$, $\int_{M}(\text{tr}_{\omega_0} \omega_{\varphi})^p e^{\lambda (2p+2)  (\delta'(p-1) + 1) \psi} \omega_0^n$ (recall that $\delta' < \frac{1}{2}$, so that $\delta'(p - 1) + 1 > \delta'$ when $p = \frac{n}{n-1}$) being bounded above implies that for $p +  \frac{1}{n-1}$, $\int_{M}(\text{tr}_{\omega_0} \omega_{\varphi})^{p +  \frac{1}{n-1}} e^{\lambda (2p +2) (\delta'(p + \frac{1}{n-1} -1) + 1) \psi} \omega_0^n$ is bounded above, then we are done by induction. Given $p > 1$, let $\tilde p = \frac{p-1}{2}$, $\beta = 2p + 2 > 2 \tilde p + 1$ and take $\tilde p, \beta$ into our inequality (\ref{maintool2}), we get
\begin{equation}
    \begin{aligned}
    &\int_{M}  e^{\lambda (2p +2) (\delta'(p + \frac{1}{n+1}-1) + 1 ) \psi} (\mathrm{tr}_{\omega_0} \omega_{\varphi})^{p + \frac{1}{n-1}} \omega_0^n\\ 
    & \leq C \int_{M}e^{-(2p+2 - \frac{n-2}{n-1} ) F- (2p + 2) \lambda \tilde \varphi - 2 \tilde p (2p+2) B}(\text{tr}_{\omega_0}\omega_{\varphi})^{p +\frac{1}{n-1}} \omega_0^n  \\
    &\leq C \int_{M} e^{\lambda (2p+2) (\delta'(p - 1) + 1)\psi} (\text{tr}_{\omega_0} \omega_{\varphi})^{p} \omega_0^n\\
    &+ C \int_{M} e^{\lambda (2p + 2) ( \delta'(p-1) + 1) \psi} (\text{tr}_{\omega_0} \omega_{\varphi})^{p - 1}  \omega_0^n\\
    &+C \int_{M}|\partial \varphi|_{\omega_0}^2  e^{\lambda (2p+2)  (\delta'(p-1) + 1) \psi} (\text{tr}_{\omega_0} \omega_{\varphi})^{p - 1} \omega_0^n\\
    &+ C \int_{M}|\partial \psi|_{\omega_0}^2 e^{\lambda (2p+2)  (\delta'(p-1) + 1) \psi} (\text{tr}_{\omega_0} \omega_{\varphi})^{p - 1} \omega_0^n.
    \end{aligned}
\end{equation}
The first two terms on the right hand side are bounded due to the inductive hypothesis and the $0$-th order bound on $F$ and $\varphi$. For the last two terms we can use the same idea as in the case where $p = 0, \beta = 2$, to bound them. Then we are done by induction. We now need the following proposition from \cite{1712.06697} to show that $\varphi_{\varepsilon}$ is locally uniformly smoothly bounded away from $E$.

\begin{Prop} \label{higherorder}
Let $K$ be a compact subset of $M$, suppose that $\mathrm{tr}_{\omega_0} \omega_{\varphi_{\varepsilon}} \in L^{p}(K)$, and $\mathrm{tr}_{\omega_{\varphi_{\varepsilon}}} \omega_0 \in L^{p}(K)$ for some $p > 3n(n-1)$, then for any $m \in \mathbb N$, $\varphi_{\varepsilon} \in C^{m}(K, \omega_0)$ where the bound only depends on $p$,  $\|\mathrm{tr}_{\omega_{\varphi_{\varepsilon}}} \omega_0\|_{L^{p}(K)}$, $\|\mathrm{tr}_{\omega_0} \omega_{\varphi_{\varepsilon}} \|_{L^{p}(K)}$, $\|\varphi_{\varepsilon}\|_{C^{0}(K)}, m$, and $(M, \omega_0)$.
\end{Prop}

\begin{Cor}
Given any compact subset $K$ of $M \backslash E$, $\varphi_{\varepsilon}$ is uniformly bounded in $C^{\infty}(K)$. 
\end{Cor}

\begin{proof}
 Given a compact subset $K \subset  M \backslash E$, the $L^{p}(K)$ bound on $\mathrm{tr}_{\omega_0}\omega_{\varphi_{\varepsilon}}$ is directly by Theorem \ref{lpbound}. Furthermore, notice that

\begin{equation} \label{xuyaozhege}
    \text{tr}_{\omega_{\varphi_{\varepsilon}}} \omega_0 = \sum_{i} \frac{1}{(g_{\varphi_{\varepsilon}})_{i \bar i}} = \sum_{i} \frac{1}{\prod_{m} (g_{\varphi_{\varepsilon}})_{m \bar m}} \prod_{j \neq i}(g_{\varphi_{\varepsilon}})_{j \bar j} = e^{-F} \prod_{j \neq i}(g_{\varphi_{\varepsilon}})_{j \bar j} \leq e^{-F}(\text{tr}_{\omega_0} \omega_{\varphi_{\varepsilon}})^{n-1}.
\end{equation}
So we are able to obtain the $L^p(K)$ bound on $\text{tr}_{\omega_{\varphi}} \omega_0$ by using the $C^{0}(M)$ bound on $F$ and Theorem \ref{lpbound} again. We conclude the proof by recalling that we showed that $\|\varphi_{\varepsilon}\|_{C^{0}(M)}$ is uniformly bounded.
\end{proof}

\section{Uniform upper bound on the Mabuchi energy and the entropy} \label{section5}
All the estimates we proved in the previous sections depend on the entropy
\begin{equation}
   \int_{M} \log \frac{\omega_{\varphi_{\varepsilon}}^n}{\omega_0^n} \frac{\omega_{\varphi_{\varepsilon}}^n}{n!}, 
\end{equation}
thus in order to obtain uniform estimates, it is essential to show that the entropies are uniformly bounded from above independent of $\varepsilon$. To achieve that, we need to study the Mabuchi energy which is closely related to the entropy. Recall that $\omega_{\varepsilon} = \eta + \varepsilon \omega_0$, and we will continue suppressing $\varepsilon$ for simplicity of notation. The Mabuchi energy is defined to be the functional
\begin{equation}
    \mathcal{M}_{\omega}(\theta) = \int_{M} \log \frac{\omega_{\theta}^n}{\omega^n} \frac{\omega_{\theta}^n}{n!} + J_{-\text{Ric}(\omega)}(\theta)
\end{equation}
where $\theta \in \mathcal{H}_{\omega}:=\{ \theta \in C^{\infty}(M) | \omega + \sqrt{-1} \partial \bar \partial \theta > 0, \sup_{M} \theta = 0 \}$. For any real $(1,1)$ form $\chi$, $J_{\chi}$ is a functional on $\mathcal{H}_{\omega}$ defined through its variation
\begin{equation}
    \frac{dJ_{\chi}}{dt} = \int_{M} \partial_{t} \theta(\text{tr}_{\omega_{\theta}}\chi -\underline{\chi}) \frac{\omega_{\theta}^n}{n!}
\end{equation}
where $\underline{\chi} = n \frac{[\chi] \cdot [\omega]^{n - 1}}{[\omega]^{n}}$. Following Chen \cite{MR1772078} (see also \cite{1801.00656} section 2.1), we have the following explicit formula for  $J_{\chi}$
\begin{equation}
    \begin{aligned}
    J_{\chi}(\theta) = -\frac{\underline{\chi}}{(n+1)!}\int_{M} \theta(\sum_{i = 0}^n \omega^{i} \wedge \omega_{\theta}^{n-i}) + \frac{1}{n!}\int_{M} \theta \chi \wedge (\sum_{i = 0}^{n-1}  \omega^i \wedge \omega_{\theta}^{n-1-i}), 
    \end{aligned}
\end{equation}
so $J_{\chi}(0) = 0$. Furthermore, the variation of the Mabuchi functional is given by
\begin{equation}
    \frac{d\mathcal{M}_{\omega}}{dt} = \int_{M} \frac{\partial \theta}{\partial t}(-\text{tr}_{\omega_{\theta}}\text{Ric}( \omega_{\theta})  + \underline{\text{Ric}(\omega})) \frac{ \omega_{\theta}^n}{n!}.
\end{equation}
The first term appearing in the Mabuchi functional is the same as what we referred to as the entropy term before except that the leading term in the integral has its denominator as $\omega^n$ while ours is $\omega_0^n$. Thus, we need to adjust the Mabuchi energy slighly to match our entropy term. We define our modified Mabuchi energy to be
\begin{equation}
    E_{\omega}(\theta) =  \int_{M} \log \frac{\omega_{\theta}^n}{\omega_0^n} \frac{\omega_{\theta}^n}{n!} +J_{\eta}( \theta).
\end{equation}
It turns out that $E_{\omega}$ only differs from $\mathcal{M}_{\omega}$ by a term equal to $\int_{M} \log \frac{\omega^n}{\omega_0^n} \frac{\omega^n}{n!}$. Observe that
\begin{equation}
    \frac{d}{dt} \int_{M} \log \frac{\omega_{\theta}^n}{\omega_0^n} \frac{\omega_{\theta}^n}{n!} = \int_{M}\frac{\partial \theta}{\partial t}(-\text{tr}_{ \omega_{\theta}}\eta - \text{tr}_{ \omega_{\theta}}\text{Ric}(\omega_{\theta} )) \frac{\omega_{\theta}^n}{n!}.
\end{equation}
So we know that 
\begin{equation}
      \frac{d E_{\omega}}{dt} = \int_{M} \frac{\partial \theta}{\partial t}(-\text{tr}_{\omega_{\theta}}\text{Ric}( \omega_{\theta})  - \underline{\eta}) \frac{ \omega_{\theta}^n}{n!}
    = \int_{M} \frac{\partial \theta}{\partial t}(-\text{tr}_{\omega_{\theta}}\text{Ric}( \omega_{\theta})  + \underline{\text{Ric}(\omega})) \frac{ \omega_{\theta}^n}{n!}.\\
\end{equation}
Thus the variation of $E_{\omega}$ is actually the same as that of the usual Mabuchi energy $\mathcal{M}_{\omega}$. We know that the potentials of cscK metrics are minimizers of $\mathcal{M}_{\omega}$ and we want to show that the same is true for $E_{\omega}$. We have just showed that $\frac{d}{dt} \mathcal{M}_{\omega}= \frac{d}{dt} E_{\omega}$. Taking a path $t \theta$ in $\mathcal{H}_{\omega}$, we have
\begin{equation}
    \mathcal{M}_{\omega}(\theta) = \int_{0}^{1} \frac{d}{dt}E_{\omega} dt + \mathcal{M}_{\omega}(0) = E_{\omega}(\theta) - \int_{M} \log \frac{\omega^n}{\omega_0^n} \frac{\omega^n}{n!}.
\end{equation}
We know that
\begin{equation}
    \mathcal{M}_{\omega}(\theta) - \mathcal{M}_{\omega}(\varphi) \geq 0
\end{equation}
for all $\theta \in \mathcal{H}_{\omega}$ and $\varphi$ such that $\omega_{\varphi} = \omega + \sqrt{-1} \partial \bar \partial \varphi$ is a cscK metric thus,
\begin{equation}
    E_{\omega}(\theta) -  E_{\omega}(\varphi) = \mathcal{M}_{\omega}(\theta) - \mathcal{M}_{\omega}(\varphi) \geq 0.
\end{equation}
So we reach the conclusion that the potentials of cscK metrics are still minimizers of $E_{\omega}$. One immediate consequence is the following lemma.
\begin{Lemma} \label{firstlemma}
Let $\varphi_{\varepsilon} \in \mathcal{H}_{\omega}$ be a smooth real-valued function such that $\omega_{\varphi_{\varepsilon}} = \omega_{\varepsilon} + \sqrt{-1} \partial \bar \partial \varphi_{\varepsilon}$ is cscK, then

\begin{equation}
E_{\omega_{\varepsilon}}(\varphi_{\varepsilon}) \leq E_{\omega_{\varepsilon}}(0) = \int_{M} \log \frac{ \omega_{\varepsilon}^n}{\omega_0^n}\frac{\omega_{\varepsilon}^n}{n!} \leq C
\end{equation}
where $C$ is independent of $\varepsilon$.
\end{Lemma}

We are now ready to show that the entropies are uniformly bounded from above. One possible approach is to first establish a uniform lower bound on the $J$ functionals since we already know from Lemma \ref{firstlemma} that $E_{\omega}(\varphi)$ is uniformly bounded from above. However, this approach requires one to study the solutions of the $J$ equations assoicated with the $J$ functionals. For more details related to this route, see  \cite{MR3981128, MR2226957, MR3561959}. Instead of following this path, we turn to an approach introduced by Dervan\cite{AFST_2016_6_25_4_919_0} which  generates explicit bounds on the functionals involved. 

To begin with, we estimate the entropy term in the Mabuchi energy using the $\alpha$-invariant, and this technique should be well-known to the experts in this field (see for example Tian \cite{MR1787650}).
\begin{Lemma} \label{xiajie}
for all $\varepsilon \in [0, 1]$, and $\alpha$ such that $\alpha > 0$ and $\alpha < \alpha(M, [\eta + \omega_0])$, we have
\begin{equation}
    \int_{M}\log \frac{\omega_{\theta}^n}{\omega_0^n}\frac{\omega_{\theta}^n}{n!} \geq - \alpha \int_{M} \theta \frac{\omega_{\theta}^n}{n!} - C
\end{equation}
where $C$ is a positive constant independent of $\varepsilon$.
\end{Lemma}

\begin{proof}
First notice that
\begin{equation}
    \int_{M} e^{-\alpha \theta} \frac{\omega_0^n}{n!} = \int_{M} e^{-\log \frac{\omega_{\theta}^n}{\omega_0^n} - \alpha \theta} \frac{\omega_{\theta}^n}{n!} \leq C
\end{equation}
where we used Lemma \ref{alphainvariant} again. Then by Jensen's inequality, 
\begin{equation}
    \int_{M}(-\log \frac{\omega_{\theta}^n}{\omega_0^n} - \alpha \theta) \frac{\omega_{\theta}^n}{n!} \leq C \int_{M} \frac{\omega^n}{n!} \leq C
\end{equation}
where $ \int_{M} \frac{\omega^{n}}{n!}$ arises because we need a probability measure in order to apply Jensen's inequality. Thus,
\begin{equation}
    \int_{M}\log \frac{\omega_{\theta}^n}{\omega_0^n}\frac{\omega_{\theta}^n}{n!} \geq - \alpha \int_{M} \theta \frac{\omega_{\theta}^n}{n!} - C.
\end{equation}

\end{proof}
We will need the following lemmas due to \cite{AFST_2016_6_25_4_919_0}, and we include brief proofs of them here for the reader's convenience.
\begin{Lemma} \label{keylemma}
We have
\begin{equation}
    -n \int_{M} \theta \omega_{\theta}^n \geq - \sum_{i = 1}^{n} \int_{M} \theta \omega^{i} \wedge \omega_{\theta}^{n-i}.
\end{equation}
\end{Lemma}

\begin{proof}
\begin{equation}
    \begin{aligned}
    -n \int_{M} \theta \omega_{\theta}^n + \sum_{i = 1}^{n} \int_{M} \theta \omega^i \wedge \omega_{\theta}^{n-i} &= \sum_{i=1}^{n} \int_{M} \theta \omega_{\theta}^{n-i} \wedge (\omega^{i} - \omega_{\theta}^i).
     \end{aligned}
\end{equation}
We only need to show that each summand is positive separately
\begin{equation}
    \begin{aligned}
    \int_{M} \theta \omega_{\theta}^{n-i} \wedge (\omega^i - \omega_{\theta}^{i}) & = \int_{M}\sqrt{-1} \partial \theta \wedge \bar \partial \theta \wedge \omega_{\theta}^{n-i} \wedge (\sum_{j = 1}^{i} \omega^{i-j} \wedge \omega_{\theta}^{j-1}) \geq 0.
    \end{aligned}
\end{equation}
\end{proof}

\begin{Lemma} \label{zhongyao}
We have
\begin{equation}
    - \frac{\underline{\eta}}{(n+1)!} \int_{M} \theta (\sum_{i = 0}^n \omega ^i \wedge \omega_{\theta}^{n-i})  \geq  - \frac{\underline{\eta}}{n! n}\int_{M} \theta \omega \wedge(\sum_{i = 0}^{n-1} \omega^i \wedge \omega_{\theta}^{n-i-1}).
\end{equation}
\end{Lemma}

\begin{proof}
We calculate
\begin{equation} 
    \begin{aligned}
    &- \frac{\underline{\eta}}{(n+1)!} \int_{M} \theta (\sum_{i = 0}^n \omega ^i \wedge \omega_{\theta}^{n-i}) \\
    &= - n \frac{ \underline{\eta}}{n(n+1)!} \int_{M} \theta \omega_{\theta}^{n} - \frac{\underline{\eta}}{(n+1)!} \int_{M} \theta \omega \wedge (\sum_{i = 0}^{n - 1} \omega \wedge \omega_{\theta}^{n - i - 1})\\
    &\geq -\frac{\underline{\eta}}{(n+1)!n} \int_{M} \sum_{i = 1}^{n} \theta \omega^{i} \wedge \omega_{\theta}^{n-i} - \frac{ \underline{\eta}}{(n+1)!} \int_{M} \theta \omega \wedge (\sum_{i = 0}^{n-1} \omega^i \wedge \omega_{\theta}^{n-i-1})\\
    &= - \frac{\underline{\eta}}{n! n}\int_{M} \theta \omega \wedge(\sum_{i = 0}^{n-1} \omega^i \wedge \omega_{\theta}^{n-i-1})
    \end{aligned}
\end{equation}
where for the inequality we used Lemma \ref{keylemma}.
\end{proof}

\begin{Th}
There exists $C$ > 0, such that
\begin{equation}
    \int_{M} \log \frac{\omega_{\varphi_{\epsilon}}^n}{\omega_0^n} \frac{\omega_0^n}{n!} \leq C.
\end{equation}
\end{Th}

\begin{proof}
Again, we will suppress $\varepsilon$ for brevity. Fix $\alpha > 0$ with $\alpha < \alpha(M + [\eta + \omega_0])$, then using Lemma \ref{xiajie} and Lemma \ref{zhongyao} we get
\begin{equation}
\begin{aligned}
    E_{\omega}(\varphi) &\geq \frac{1}{2}\int_{M} \log \frac{\omega_{\varphi}^n}{\omega_0^n} \frac{\omega_{\varphi}^n}{n!}- \frac{\alpha}{2} \int_{M} \varphi \frac{\omega_{\varphi}^n}{n!} - \frac{\underline{\eta}}{(n+1)!}\int_{M} \varphi  (\sum_{i=0}^{n}  \omega^i \wedge  \omega_{\varphi}^{n-i})\\
    &+ \frac{1}{n!}\int_{M} \varphi \eta \wedge(\sum_{i=0}^{n-1} \omega^i \wedge \omega_{\varphi}^{n-1-i})- C\\
    & \geq \frac{1}{2}\int_{M} \log \frac{\omega_{\varphi}^n}{\omega_0^n} \frac{\omega_{\varphi}^n}{n!}  - \frac{1}{n!}\int_{M} \varphi(\frac{\alpha}{2n} \omega -\eta + \frac{1}{n}\underline{\eta} \omega) \wedge(\sum_{i= 0}^{n-1} \omega^{i} \wedge \omega_{\varphi}^{n-1-i}) - C
\end{aligned}
\end{equation}
where for the first inequality we used Lemma \ref{zhongyao}, and the second inequality we used Lemma \ref{keylemma} again. Notice that
\begin{equation}
\begin{aligned}
     \frac{\alpha}{2n} \omega -\eta + \frac{ 1}{n} \underline{\eta} \omega = ( \frac{\alpha}{2n} - 1 + \frac{1}{n}\underline{\eta})\eta +  \varepsilon (\frac{\alpha}{2n} + \frac{1}{n}\underline{\eta})  \omega_0 > 0
\end{aligned}
\end{equation}
for $\varepsilon > 0$ small enough because $\underline{\eta} = n \frac{ [\eta] \cdot [\omega_{\varepsilon}]^{n-1}
}{[\omega_{\varepsilon}]^{n}}\to n$ as $\varepsilon \to 0$. Thus for $\varepsilon$ sufficiently small, we have

\begin{equation}
\begin{aligned}
    E_{\omega}(\varphi) 
    & \geq \frac{1}{2}\int_{M} \log \frac{\omega_{\varphi}^n}{\omega_0^n} \frac{\omega_{\varphi}^n}{n!}  - \int_{M} \varphi(\frac{\alpha}{2n} - \eta + \frac{1}{n}\underline{\eta}\omega) \wedge(\sum_{i= 0}^{n-1} \omega^{i} \wedge \omega_{\varphi}^{n-1-i}) - C\\
    &\geq \frac{1}{2}\int_{M} \log \frac{\omega_{\varphi}^n}{\omega_0^n} \frac{\omega_{\varphi}^n}{n!}  - C,
\end{aligned}
\end{equation}
since we assume that $\varphi \leq 0$. We conclude the proof by recalling that $E_{\omega} (\varphi) \leq E_{\omega}(0) \leq C$ according to Lemma \ref{firstlemma}.
\end{proof}

\section{Convergence} \label{lastsection}
Now we use the estimates developed in previous sections to show that away from $E$, $\omega_{\varphi_{\varepsilon}}$ converges to the singular K\"{a}hler Einstein metric $\omega_{\text{KE}}$ on $M \backslash E$ given by Theorem \ref{singular}. Fix a connected compact subset $K \subset M$,  and a subsequence $\varepsilon_{i}$ such that $\varphi_{\varepsilon_{i}}$ converges smoothly to $\varphi_{\infty}$, and $F_{\varepsilon_{i}} $ converges smoothly to $F_{\infty}$ (such a subsequence always exists by Arzela-Ascoli). Recall that our equations read
\begin{equation}
    \begin{aligned}
    & \frac{(\omega_{\varepsilon_{i}} +\sqrt{-1}\partial \bar \partial  \varphi_{\varepsilon_{i}})^n}{\omega_0^n} = e^{F_{\varepsilon_{i}}},\\
    &\Delta_{\omega_{\varphi_{\varepsilon_{i}}}} F_{\epsilon_{i}}= -\text{tr}_{\omega_{\varphi_{\varepsilon_{i}}}}\eta - \underline{R_{\varepsilon_{i}}}.
    \end{aligned}
\end{equation}
Let us compute 
\begin{equation}
    \begin{aligned}
     \int_{K} |\partial(F - \varphi)|^2_{\omega_{\varphi}} \omega_{\varphi}^{n} &\leq  \int_{M} |\partial(F - \varphi)|^2_{\omega_{\varphi}} \omega_{\varphi}^{n}\\ 
    &= - \int_{M} (F - \varphi) \Delta_{\omega_{\varphi}} (F - \varphi) \omega_{\varphi}^{n} \\
    &= (\underline{R} + n) \int_{M} (F - \varphi) \omega_{\varphi}^n - \varepsilon \int_{M} (F - \varphi) \text{tr}_{\omega_{\varphi}} \omega_0 \omega_{\varphi}^n  \\
    &\leq C(\underline{R} + n) \int_{M} \omega_{\varphi}^n + C \varepsilon n \int_{M}  \omega_{\varphi}^{n-1} \wedge \omega_{0} \to 0
    \end{aligned}
\end{equation}
where for the last line we used the $C^{0}$ estimates on $F$ and $\varphi$. So $\partial F_{\infty} = \partial \varphi_{\infty}$ on $K$ which implies that $F_{\infty} =  \varphi_{\infty} + C$ on $K$. Furthermore,
\begin{equation}
    \text{Ric}(\omega_{\varphi_{\infty}}) = \mathrm{Ric}(\omega_0) - \sqrt{-1} \partial \bar \partial F_{\infty}  = -\eta - \sqrt{-1} \partial \bar \partial \varphi_{\infty} = - \omega_{\varphi_{\infty}}
\end{equation}
on $K$ and thus on all of $M \backslash E$. By Theorem \ref{singular}, the uniqueness of such singular K\"{a}hler Einstein metric, all convergent subsequences have to converge to the same limit. If for each sequence we have a further subsequence converging to the same limit, then the sequence itself has to converge.

\section*{Acknowledgement}
The author is grateful to his thesis advisor Ben Weinkove for his continued support, encouragement, countless discussions around these results as well as many valuable comments on the manuscript. The author would also like to thank Jian Song for pointing out an error in the original proof of Theorem 2.5. 
\medskip
\printbibliography

\end{document}